\documentclass{amsart}
\usepackage[utf8]{inputenc}
  \usepackage{amsmath}
\usepackage{enumerate}
  \usepackage{amssymb}
  \usepackage{amsthm}
  \usepackage{bm}
  \usepackage{graphicx}
  \usepackage{color}
  \usepackage{bigints}
\usepackage{esint}
\usepackage{nccmath}
\usepackage{relsize}
\usepackage{accents}
\usepackage{epsfig}
\usepackage{amsmath,amssymb}
\hoffset - 0.5 in
\newtheorem{Th}{Theorem}[section]
\newtheorem{lem}[Th]{Lemma}

\theoremstyle{definition}
\newtheorem{Def}[Th]{Definition}

\newtheorem{Cor}[Th]{Corollary}
\newtheorem{Prop}[Th]{Proposition}

\theoremstyle{remark}
\newtheorem{rem}[Th]{Remark}

\numberwithin{equation}{section}

\newcommand{\tend}[3][]{\xrightarrow[#2\to#3]{#1}}

\newcommand{\egdef}{\stackrel{\textrm {def}}{=}}
\newcommand{\ds}{\displaystyle}

\newcommand{\R}{\mathbb{R}}

\newcommand{\T}{\mathbb{T}}

\begin{document}

\title{Calculus of Generalized Riesz Products}

\author{e. H. el Abdalaoui}
\address{Normandie University, University of Rouen
  Department of Mathematics, LMRS  UMR 60 85 CNRS\\
Avenue de l'Universit\'e, BP.12
76801 Saint Etienne du Rouvray - France .}
\email{elhoucein.elabdalaoui@univ-rouen.fr}
\urladdr{http://www.univ-rouen.fr/LMRS/Persopage/Elabdalaoui/}
\author{M. G. Nadkarni}
\address{Department of Mathematics, University of Mumbai, Vidyanagari, Kalina,  Mumbai, 400098, India}
\email{mgnadkarni@gmail.com}
\urladdr{http://insaindia.org/detail.php?id=N91-1080}


\dedicatory{}


\begin{abstract}
In this paper we discuss generalized Riesz products bringing into consideration $H^p$ theory, the notion of Mahler measure, and the zeros of polynomials appearing in the generalized Riesz product. Formula for Radon-Nikodym derivative between two generalized Riesz product is established under suitable conditions. This is then used to formulate a Dichotomy theorem and prove a conditional version of it. A discussion involving flat polynomials is given.
\end{abstract}

\maketitle

\section{Introduction.}
 Generalized Riesz products considered in this paper are the ones defined in \cite{Host-Mela-Parreau},\cite{Nadkarni1} where one of the  aims was to describe the spectrum of a non-singular rank one map as a generalized Riesz product. Generalized Riesz product has  remained only in the state of definition although much deep work has appeared over last two decades on special generalized Riesz products arising in the spectral study of measure preserving rank one maps of ergodic theory.\\

 The purpose of this paper is to set forth some basic facts of generalized
 Riesz products and bring into play Hardy class theory to discuss some of the problems arising in the subject.
 It is surprising that one can garner so much information simply from the fact that $L^2$ norm of the trigonometric polynomials
 appearing in a generalized Riesz product is one. These facts are discussed in section 2 and 3 of the paper.
  Section 4 gives a formula for Radon-Nikodym derivative of two generalized Riesz products.
  In section 5 we discuss a conditional dichotomy result. Connection with flat polynomials is discussed in section 6 and
  a result on zeros of polynomials under consideration is given in section 7.\\

In the rest of this section we give some background material.\\

\paragraph{\textbf{Riesz Products.}} Consider a trigonometric series
$$\sum_{n=-\infty}^{+\infty}a_n z^n,~~z \in S^1,$$
where $S^1 = \{z ~~:~~ | z |= 1\}$, the circle group. If we ignore those terms whose coefficients are zero, then we can write the trigonometric series as
$$\sum_{k=-\infty}^{+\infty}a_{n_k}z^{n_k},~~z \in S^1.$$
Now if $|\frac{n_{k+1}}{n_k}|> q$, $\forall k$, for some $q > 1$, then the series is said to be lacunary. The
convergence questions for a lacunary trigonometric are answered by:
\begin{Th}[{\cite[p.203]{Zygmund}}]\label{thZ}
A lacunary trigonometric series converges on a set of positive Lebesgue measure if and only if its coefficients form an $\ell^2$ sequence. If the coefficients are not square summable then the lacunary trigonometric series is not a Fourier series (of any $L^1(S^1, dz)$ function).
\end{Th}
Next we need the notion of dissociated polynomials. Consider the following two products:
\begin{eqnarray*}
(1+z)(1+z) &=& 1+z+z +z^2 = 1 +2z +z^2,\\
(1+z)(1+z^2) &=& 1+z+z^2+z^3.
\end{eqnarray*}
In the first case we group terms with the same power of $z$, while in the second case all the powers of $z$ in the formal expansion are distinct. In the second case we say that the polynomials $1 + z$ and $1 + z^2$ are dissociated. More generally we say that a set of trigonometric polynomials is dissociated if in the formal expansion of product of any finitely many of them, the powers of $z$ in the non-zero terms are all distinct. (see section 2 for a detailed definition).\\

Now consider the infinite Product due to F. Riesz:  (\cite[p.208]{Zygmund})
$$\prod_{k=1}^{+\infty} \Bigl(1+a_k \cos(n_k x)\Bigr),~~-1\leq a_k \leq
1,~~\frac{n_{k+1}}{n_k} \geq 3.$$
Each term of this product is non-negative and integrates to 1 with respect to the normalized Lebesgue measure on the circle group. We rewrite this product as
$$\prod_{k=1}^{+\infty} \Bigl(1+\frac{a_k}2 \bigl(z^{n_k}+z^{-n_k}\bigr) \Bigr). $$

Because of the Lacunary nature of the $n_k$'s, the polynomials
$$1+\frac{a_k}2 \bigl(z^{n_k}+z^{-n_k}\bigr),~~k=1,2,\cdots$$
are dissociated. If we expand the finite product $\prod_{k=1}^{L} \Bigl(1+\frac{a_k}2 \bigl(z^{n_k}+z^{-n_k}\bigr) \Bigr)$, we get a finite sum of the type
$$1+\sum_{\overset{k \neq 0}{k=-M}}^{M}\gamma_k z^{m_k},$$
and for the infinite product we get the series
$$\sum_{k=-\infty}^{+\infty}\gamma_k z^{m_k},$$
both sums being formal expansions of the corresponding products. Since the finite products are non-negative and integrate to $1$, they are probability densities and the corresponding probability measures converge weakly to a probability measure, say $\mu$, whose Fourier-Stieltjes series is the formal expansion of the infinite product. The main theorem about Riesz products is
\begin{Th}[{\cite[p.209]{Zygmund}}]\label{zyg2}
The Riesz product $\prod_{k=1}^{+\infty} \Bigl(1+\frac{a_k}2 \bigl(z^{n_k}+z^{-n_k}\bigr) \Bigr)$, $-1 \leq a_k \leq 1,~~\frac{n_{k+1}}{n_k} \geq 3, \forall k.$ represents a continuous measure $\mu$ on $S^1$ which is absolutely continuous or singular with respect to the Lebesgue measure on $S^1$ according as the sequence $a_k,k=1,2,\cdots$ is in $\ell^2$ or not. The finite products
$$\prod_{k=1}^{L} \Bigl(1+\frac{a_k}2 \bigl(z^{n_k}+z^{-n_k}\bigr) \Bigr),~~L=1,2,\cdots$$ converge to $\frac{d\mu}{dz}$ a.e. (dz) as $L \longrightarrow +\infty.$
\end{Th}

We will improve this theorem using Hardy class theory later.\\

The above account about Riesz products is based on parts of chapter V of Zygmund \cite{Zygmund}. The original four page paper of F. Riesz appeared nearly 95 years ago in 1918 \cite{Riesz}, and has led to much further work. The aim of the paper was to give a continuous function of bounded variation whose Fourier-Stieltjes coefficients do not tend to zero.\\

\paragraph{\textbf{Connection with Ergodic Theory.}} In \cite{Ledrappier} F. Ledrappier showed that Riesz products appear as maximal spectral type of a class of measure preserving transformation. We will assume that the reader is familiar with the stacking method of constructing measure preserving transformations. Consider a sequence stacks $\Sigma_n$, $n = 1, 2, \cdots$ of pairwise disjoint intervals, beginning with the unit interval as $\Sigma_1$. Each stack comes equipped with the usual linear maps among its element. Let $h_n$ be the height of $\Sigma_n$, $n = 1, 2, \cdots$ . For each $n \geq 2$, the stack $\Sigma_n$ is obtained from $\Sigma_{n-1}$ by dividing $\Sigma_{n-1}$ into two equal parts and adding a finite number, say $a_{n-1}$, of spacers on the left subcolumn. The top of the left subcolumn (after adding spacers) is mapped linearly onto the bottom of the right subcolumn, and the resulting new stack is $\Sigma_n$. If $T$ is the measure preserving transformation given by this system of stacks,  then, as shown by Ledrappier, the associated unitary operator $U_T$ has simple spectrum whose maximal spectral type,( except possibly for some discrete part), is given by the weak$*$ limit of probability measures
$$\mu_L=\prod_{n=1}^{L}\Bigl|\frac1{\sqrt{2}}\bigl(1+z^{h_n+a_n}\bigr)\Bigr|^2 dz=
\prod_{n=1}^{L}\Bigl|1+\frac1{2}\bigl(z^{h_n+a_n}+z^{-h_n-a_n}\bigr)\Bigr|^2 dz, L=1,2,\cdots$$
as $L\longrightarrow +\infty$. We write this measure formally as
$$\mu =\prod_{n=1}^{+\infty}\Bigl|\frac1{\sqrt{2}}\bigl(1+z^{h_n+a_n}\bigr)\Bigr|^2 ~~~~\eqno (A).$$

More generally, consider a rank one measure preserving transformation made of a sequence of stacks $\Sigma_n$, $n = 1, 2,\cdots$ with $h_n$ as the height of $\Sigma_n$, and $\Sigma_1$ being the unit interval. For $n \geq 2$, $\Sigma_n$ is obtained from $\Sigma_{n-1}$ by dividing $\Sigma_{n-1}$ into $m_{n-1}$ equal parts, and placing a certain number $a_j^{(n-1)}$ of spacers on the $j^{\textrm{th}}$ subcolumn, $1 \leq j \leq m_{n-1}-1$. Let $T$ denote the resulting measure preserving transformation and $U_T$ the associated unitary operator. Here again the operator $U_T$ has simple spectrum and the maximal spectral type (except possibly some discrete part) is given by the weak$*$ limit of the probability measures \cite{Host-Mela-Parreau},\cite{Nadkarni1},\cite{KlemesR}:
$$\mu_n=\prod_{k=1}^{n}\frac1{m_k}\Bigl|\Bigl(1+\sum_{j=1}^{m_{k}-1}z^{jh_k+\sum_{i=1}^{j}a_i^{(k)}}\Bigr)\Bigr|^2 dz,~~n=1,2,3,\cdots $$
We denote this weak$*$ limit $\mu$ by the infinite product:
$$\mu=\prod_{k=1}^{+\infty}\frac1{m_k}\Bigl|\Bigl(1+\sum_{j=1}^{m_{k}-1}z^{jh_k+\sum_{i=1}^{j}a_i^{(k)}}\Bigr)^2\Bigr|  \eqno (B)$$

Ornstein \cite{ornstein} has constructed a family of rank one measure preserving maps which are mixing. Bourgain \cite{bourgain}
has shown that almost all of these rank one map have singular spectrum. It is not known if there exists a rank one measure preserving map whose maximal spectral type has a part which is absolutely continuous with respect to the Lebesgue measure on $S^1$. This question is naturally related to a question of Banach (in The Scottish Book) which asks if there is a measure preserving transformation $T$ on the Real line (with Lebesgue measure) which admits a function $f$ such that $f \circ T^n$, $n = 1, 2, \cdots$ are orthogonal and span $L^2(\R).$

Let $\Sigma_n,$ $n = 1, 2,\cdots$ and $T$ be as above, and let $\phi$ be a function of absolute value one which is constant on interval of $\Sigma_n$ except the top piece, $n = 1, 2, \cdots$ . It is known that the unitary operator $V_{\phi} = \phi\cdot U_T$ also has multiplicity one and its maximal spectral  type  is  continuous  whenever $\phi$  is  not  a  coboundary.	It  is  given  by  the weak$*$ limit of a sequence of probability measures given by:	

$$\mu_L=\prod_{n=1}^{L}\Bigl|P_k(z)|^2 dz,$$
where $P_k'$s are polynomials of the type:
$$P_k(z)=c_0^{(k)}+\sum_{j=1}^{m_k-1}c_j^{(k)} z^{jh_k+\sum_{i=1}^{j}a_i^{(k)}},~~~~ \sum_{j=0}^{m_k-1}|c_j^{(k)}|^2=1.$$
The constants $c_j^{(k)}$ are determined by the $m_k$'s  and the function $\phi$. We may write	 this weak$*$ limit as
$$ \mu_{\phi}=\prod_{n=1}^{+\infty}\Bigl|P_k(z)|^2 \eqno (C).$$

Note that in all the products (A), (B), and (C) there is no lacunarity condition imposed on the powers of $z$ from outside. The gap between two consecutive nonzero terms of the polynomials are determined by parameters of construction, and need not be lacunary. In the rest of this paper we will, for most part, dispense with the dynamical origin of the measures of the type (A), (B) and (C) and discuss a larger class of measures, called generalized Riesz product, which include these measures.

\section{Generalized Riesz Products and their  Weak Dichotomy.}

 In this section we introduce generalized Riesz products and derive a weak dichotomy result
about  infinite product of polynomials associated to it. This also yields conditions for absolute continuity of the generalized Riesz product.
\begin{Def}\label{def1}
Let $P_1, P_2, \cdots,$ be a sequence of trigonometric polynomials such that
\begin{enumerate}[(i)]
\item for any finite sequence $i_1< i_2 < \cdots < i_k$ of natural numbers
$$\bigintss_{S^1}\Bigl| (P_{i_1}P_{i_2}\cdots P_{i_k})(z)\Bigr|^2dz = 1,$$
where $S^1$ denotes the circle group and $dz$ the normalized Lebesgue measure on $S^1$,
\item for any infinite sequence $i_1 < i_2 < \cdots $ of natural numbers the weak limit of the measures
$\mid (P_{i_1}P_{i_2}\cdots P_{i_k})(z)\mid^2dz, k=1,2,\cdots $ as $k\rightarrow \infty$ exists,
\end{enumerate}
then the measure $\mu$ given by the weak limit of $\mid (P_1P_2\cdots P_k)(z)\mid^2dz $ as $k\rightarrow \infty$
is called generalized Riesz product of the polynomials $\mid P_1\mid^2,
 \mid P_2\mid^2,\cdots$ and denoted by
  $$\displaystyle  \mu =\prod_{j=1}^\infty \bigl| P_j\bigr|^2  \eqno (1.1).$$

\end{Def}

For an increasing sequence $k_1 < k_2 < \cdots $ of natural numbers the product \linebreak $\prod_{j=1}^\infty |P_{k_j}|^2$
makes sense as the weak limit of probability measures \linebreak $| (P_{k_1}P_{k_2}\cdots P_{k_{n}})(z)|^2dz$ as $n\rightarrow \infty$. It depends on the sequence $k_1 < k_2\cdots$, and  called a subproduct of the given generalized Riesz product.

 Since the object under consideration is the generalized Riesz product \linebreak $\prod_{j=1}^\infty| P_j|^2$, without loss of generality we assume that the polynomials $P_j, j =1,2,\cdots$ are analytic with positive constant term. Their domain of definition will mainly be the circle group, but with option to look at then as functions on the complex plane. Since $\ds \int_{S^1}| P_j|^2(z)dz =1$, the sum of the squares of the absolute values of coefficients of
 $P_j$ is one, and so each coefficient of $P_j$ is at most one in absolute value. Let $a_0^{(j)}$
 denote the constant term of $P_j$, which is positive by assumption. The sequence of products $\prod_{j=1}^na_0^{(j)}, n=1,2,\cdots$ is non-increasing, and so has a limit which is either zero or some positive constant which can be at most 1. The case when this constant is one is obviously the trivial case when each $P_j$ is the constant 1.

 Consider the sequence of polynomials $ S_n \stackrel{\textrm {def}}{=} \prod_{j=1}^nP_j, n=1,2,\cdots$ (without the absolute value squared). For each $n$, let  $b_0^{(n)} \stackrel{\textrm {def}}{=} \prod_{j=1}^n a_0^{(j)}$ denote the constant term of $S_n$. Write $ b =\lim_{n\rightarrow \infty}b_0^{(n)}$. We have the following  weak dichotomy theorem for generalized Riesz products.

 \begin{Th}\label{th1}
 If $b=0$, the  sequence of polynomials $S_n =\prod_{i=1}^nP_i$, $i =1,2,\cdots$ converges to zero weakly in $L^2(S^1,dz)$.  If $b$ is positive it converges in $L^1(S^1, dz)$ (and in $H^1$) norm to a non-zero function $f$ which is also in $H^2$ with $H^2$ norm at most 1, $\log(|f|)$ has finite integral.
\end{Th}

\begin{proof}
 Assume that   $b =0 $. We show that the sequence $S_n, n=1,2,\cdots$ has zero
 weak limit as functions in  $L^2(S^1, dz)$. Assume that a subsequence $S_{k_n}, n=1,2,\cdots$ converges weakly
 $f \in L^2(S^1,dz)$. We show that $f$ is the zero function. By choosing a further subsequence if necessary we can assume without  any loss that the constant term of $\frac{S_{k_{n+1}}}{S_{k_n}}, n=1,2,\cdots$ goes to zero as $n\rightarrow \infty$.  Since $b=0$, the zeroth Fourier coefficient of $f$ is zero. Since each $S_n$ is an analytic
 trigonometric polynomial, the negative Fourier coefficients of $f$ are all zero.
 Assume now that for $0\leq  j < l$,  $\displaystyle b_j = \int_{S^1}z^{-j}f(z)dz =0$. Then, given $\epsilon$, for large enough $m$, $\displaystyle \Bigl| \int_{S^1} z^{-j}S_{k_m} dz \Bigr| < \epsilon$, for $0 \leq j< l$, and moreover the constant term of $\frac{S_{k_{m+1}}}{S_{k_m}}$
 is less than $\epsilon $. For $n > m$,
  $$\prod_{j=1}^{k_n}P_i = \prod_{j=1}^{k_m}P_j\prod_{j=k_m+1}^{k_n}P_j.$$
  Since $P_j$'s are one sided trigonometric polynomials, it is easy to see from this that $\displaystyle \Bigl| \int_{S^1}z^{-l}S_{k_n}(z)dz\Bigr| \leq (l+1)\epsilon$. Since this holds for all
  $n > m$ we see that $\displaystyle \int_{S^1}z^{-l}f(z)dz =0$. Induction completes the proof.

 Assume now that $b > 0$.
 Then $a_0^{(n)}$ as well as $ \prod_{j=m+1}^n a_0^{(j)}, m < n,$ converge to 1 as $m,n \rightarrow \infty$. Since $L^2(S^1, dz)$ norm of all the finite  products is one,
 \begin{eqnarray}
  P_n, \frac{S_n}{S_m} \rightarrow 1 {~~\textrm{in}}~~ L^2(S^1, dz),~~ m < n,~~  {~~\textrm{as}}~~ m,n \rightarrow \infty.
  \end{eqnarray}
   Moreover by Cauchy-Schwarz inequality
 $$\Bigl\| S_n - S_m\Bigr\|_1 = \Bigl\| S_m\Bigl(1 - \prod_{m+1}^nP_i\Bigr)\Bigr\|_1 \leq \Bigl\| S_m\Bigr\|_2 \Bigl\|1- \prod_{j=m+1}^nP_j\Bigr\|_2$$
$$\rightarrow 0 ~~{\rm{as}}~~ m, n \rightarrow \infty.$$

\noindent{}Thus the sequence of analytic polynomials $S_n, n=1,2,\cdots$ converges in $L^1(S^1, dz)$ to a function which we denote by $f$, and view it also as a function in the Hardy class $H^1$. A subsequence of $S_n, n=1,2,\cdots$ converges to $f$ a.e (with respect to the Lebesgue measure of $S^1)$, whence, over the same subsequence $S^2_n, |S_n|^2, n=1,2,\cdots$ converge to $f^2$ and
$|f|^2$, respectively. Since $\|S_n\|_2 =1, n =1,2,\cdots$, by Fatou's lemma we conclude
that $f$ is square integrable with $L^2(S^1, dz)$ norm at most 1. Thus $f$ is in $H^2$, and $\log \mid f\mid$ has finite integral.\\
\end{proof}

We do not  know if the $L^2(S^1, dz)$ norm of $f$ is 1, equivalently, if $S_n, n=1,2,\cdots$  converges to $f$ in
$L^2(S^1,dz)$. We give some sufficient conditions under which this holds.  Let $S_n = \sum_{j=0}^{m_n}b_j^{(n)}z^j$,
where $m_n$ is the degree of the trigonometric polynomial $S_n$. Now
$$b_j^{(n)} = \int_{S^1}z^{-j}S_n(z)dz \rightarrow \int_{S^1}z^{-j}f(z)dz \stackrel{\textrm{def}}{=} b_j.$$
The series $\sum_{j=0}^\infty b_jz^j$ is the Fourier series of $f$ and we call this series the formal expansion of
$ \prod_{j=1}^\infty P_i$. Since $b$ is positive, the infinite product $ \prod_{j=n+1}^\infty a_0^{(j)}$ is also positive,
so the  infinite product $ \prod_{j=n+1}^\infty P_j$ has a formal expansion which we denote by
$\sum_{j=0}^\infty c_j^{(n)}z^j$. Note that $c_0^{(n)} =\prod_{j=n+1}^\infty a_{0}^{(j)} \longrightarrow 1$
as $n\longrightarrow \infty$, as a result $\sum_{j=1}^\infty|c_j^{(n)}|^2 \leq 1- (c_0^{(n)})^2 \longrightarrow 0$ as
$ n \rightarrow \infty$.\\

At this point, let us recall the following important notion in the Riesz product theory.

\begin{Def}\label{def2}  Finitely many trigonometric polynomials $q_0,q_1,\cdots,q_n$,\linebreak $q_j=\sum_{i=-N_j}^{N_j} d_i^{(j)}z^i$,$j=0,1,2,\cdots,n$ are said to be dissociated if in their product \linebreak $q_0(z)q_1(z)\cdots q_n(z)$, (when expanded formally, i.e., without grouping terms or canceling identical terms with opposite signs), the powers $i_0+i_1+\cdots+i_n$ of $z$ in non-zero terms
$$d_{i_0}^{(0)}d_{i_1}^{(1)}\cdots d_{i_n}^{(n)}z^{i_0+i_1+\cdots+i_n}$$ are all distinct.\\

A sequence $q_0,q_1,\cdots,$ of trigonometric polynomials is said to be dissociated if for each $n$ the polynomials
$q_0,q_1,\cdots,q_n$ are dissociated.
\end{Def}

Suppose now that the polynomials $P_1, P_2, \cdots$ (without the squares) appearing in generalized Riesz product 1.1 are dissociated. Then, whenever $b_j^{(n)}$ is a non-zero coefficient in the expansion of $S_n$, $ b_j^{(l)} = b_j^{(n)}\frac{b_0^{(l)}}{b_0^{(n)}}$ for all $l \geq n$. Thus, if the polynomials $P_j, j =1,2,\cdots$ are dissociated, then we see on letting $ l \rightarrow \infty$ that $b_j = b_{j}^{(n)}c_0^{(n)}$, provided $b_j^{(n)} \neq 0$. We therefore have for any $n$
 $$\sum_{j=0}^{m_n}| b_j|^2 \geq \sum_{j=1}^{m_n}\mid b_j^{(n)}\mid^2 (c_0^{(n)})^2 = (c_0^{(n)})^2 \longrightarrow 1$$ as $ n\rightarrow \infty$. Thus $ f$ has $L^2(S_1,dz)$ norm 1.
 We have proved:

 \begin{Th}\label{th2}
 If the polynomials $P_n, n=1,2,\cdots$ are dissociated and
 b is positive then the partial products $S_n, n=1,2,\cdots$ converge in $ H^2$ to a non-zero function
 $f$ and the generalized product $\prod_{j=1}^\infty |P_j|^2$ is the measure $|f|^2dz$.
 Further, $\ds \int_{S^1}\log (| f(z)|) dz$ is finite.
 \end{Th}

  If we replace the condition that the polynomials $P_n, n=1,2,\cdots$ are dissociated by the condition that coefficients of the polynomials $P_n, n=1,2,\cdots$ are all non-negative, then it is easy to verify that for $0 \leq k \leq m_n$,
  $$ b_k \geq c_0^{(n)}b_k^{(n)}+ b_0^{(n)}c_k^{(n)},$$
  whence
  $$\sum_{k=0}^\infty\mid b_k\mid^2 \geq (\sum_{k=0}^{m_n}\mid b_k^{(n)}\mid^2)\mid c_0^{(n)}\mid^2=1\cdot \mid c_0^{(n)}\mid^2 \rightarrow 1,$$ as $n \rightarrow \infty$. Thus, if the coefficients of all the $P_n, n=1,2,\cdots$ are non-negative, and if $b =b_0 > 0$, we necessarily have convergence of $S_n, n=1,2,\cdots $ in $H^2$.\\

  We continue with the assumption that $b$ is positive, but no more assume that the polynomials $P_n, n=1,2,\cdots$ are dissociated or have non-negative coefficients. Fix $n$, and let $1 \leq j \leq m_n$, then
  $$\sum_{j=0}^\infty b_jz^j = \Bigl(\sum_{i=0}^{m_n}b_i^{(n)}z^i\Bigr) \Bigl(\sum_{k=0}^\infty c_k^{(n)}z^k\Bigr).$$
 This gives any $j \geq 0$,
$$ b_j = b_j^{(n)}c_0^{(n)} + \sum_{i=0}^{j-1}b_i^{(n)}c_{j-i}^{(n)}.$$
Hence, for any $j \geq 1$,
$$\mid b_j - b_j^{(n)}c_0^{(n)}\mid^2 \leq \Bigl(\sum_{i=0}^{j-1}\mid b_i^{(n)}\mid^2\Bigr)\Bigl(\sum_{i=0}^{j-1}\mid c^{(n)}_{j-i}\mid^2\Bigl)\leq \sum_{j=1}^\infty \mid c^{(n)}_j\mid^2 $$
  $$\leq 1 - (c_0^{(n)})^2  \longrightarrow 0 ~~{\rm{as}} ~~ n\rightarrow \infty$$
  Assume now that $m_n(1-c_0^{(n)}) \longrightarrow 0$ as $n\rightarrow \infty$
  Then $\sum_{j=0}^{m_n}\mid b_j - b_j^{(n)}c_0^{(n)}\mid^2  \rightarrow 0$ as $n\rightarrow \infty.$
Another use of the assumption that $m_n(1-c_0^{(n)}) \longrightarrow 0$ as $n\rightarrow 0$ allows us to conclude that
$\sum_{j=0}^{m_n}| b_j - b_j^{(n)}|^2 \longrightarrow 0$ as $n\rightarrow \infty$. Since $\sum_{j=1}^{m_n}| b_j^{(n)}|^2 =1$, we conclude that $\sum_{j=1}^\infty | b_j|^2 = 1$, so that $L^2(S^1,dz)$ norm of $f$ is one and $S_n, n=1,2,\cdots$ converges to $f$ in $H^2$. We have proved:
\begin{Th}\label{th3}
If $b$ is positive and $m_n(1 -c_0^{(n)}) \longrightarrow 0$ as $n\rightarrow \infty$ then
$S_n \rightarrow f$ in $H^2$ and $|f|^2dz$ is the generalized Riesz product $\prod_{j=1}^\infty|P_j|^2$.
Moreover $\log \big(|f|\big)$ has finite integral.
\end{Th}
Our calculus can be interpreted as follows. Put $$B=(b_j)_{j=0}^{+\infty},~~ B^{(n)}=(b_0^{(n)},b_1^{(n)},\cdots,b_{m_n}^{(n)},0,0,\cdots ) {\textrm{~~and~~}}
 C^{(n)}=(c_j^{(n)})_{j=0}^{+\infty},$$
then
  $$\widehat{B}(z)=\sum_{j=0}^\infty b_jz^j = \Bigl(\sum_{i=0}^{m_n}b_i^{(n)}z^i\Bigr)\Bigl(\sum_{k=0}^\infty c_k^{(n)}z^k\Bigr)=
  \widehat{B^{(n)}}(z)\widehat{C^{(n)}}(z)=\widehat{{B^{(n)}*C^{(n)}}}(z).$$
Therefore
$$\Bigl\|B-B^{(n)}\Bigr\|_2=\Bigl\|{B^{(n)}*C^{(n)}}-B^{(n)}\Bigr\|_2=
\Bigl\|B^{(n)}*\Bigl(C^{(n)}-(1,0,0,\cdots)\Bigr)\Bigr\|_2.$$
Hence
$$\Bigl\|B-B^{(n)}\Bigr\|_2 \leq \Bigl\|B^{(n)}\Bigr\|_1
\Bigl\|C^{(n)}-(1,0,0,\cdots)\Bigr\|_2.$$
It follows under our condition $(m_n(1 -c_0^{(n)}) \longrightarrow 0$ as $n\rightarrow \infty$) that
$$\Bigl\|B^{(n)}\Bigr\|_1
\Bigl\|C^{(n)}-(1,0,0,\cdots)\Bigr\|_2 \longrightarrow 0.$$
\begin{Cor}\label{cor1}
If $b > 0$ then there is always a subproduct  $\prod_{k=1}^\infty P_{n_k}$ for which the condition of the above theorem is satisfied, so that if $b > 0$ holds, then a subproduct $\prod_{k=1}^\infty \mid P_k(z)\mid^2 $ has the same null sets as Lebesgue measure.
\end{Cor}
\begin{proof}
Put $k_1 =1$ and $ P_{k_1} = P_1$. Let $m_1$ be the degree of $P_{k_1}$. Since $b > 0$,
$c_0^{(n)} \rightarrow 1$ as $n\rightarrow \infty$. Choose $k_2 > k_1$ such that $m_1 (1 - c_0^{(k_2)})  \leq \frac{1}{2}$. Consider $P_{k_1}\cdot P_{k_2+1}$. Suppose its degree is $m_2$. Choose $k_3 > k_2$ such that
$m_2(1-c_0^{(k_3)})\leq \frac{1}{4} $. Assume that we have chosen $k_1< k_2 < \cdots< k_{l-1}$ such that for any $i, 1\leq i \leq l-2 $ if $m_i$ is the degree of $P_{k_1}P_{k_2+1}\cdots P_{k_i+1}$, then $$m_i(1 - c_0^{(k_{(i+1)})}) \leq \frac{1}{2^i}.$$ Choose
$k_{l}> k_{1-1}$ such that $$m_{l-1}(1 - c_0^{(k_l)}) \leq \frac{1}{2^{l-1}}.$$
Thus we have inductively chosen a sequence  $k_1 < k_2 < k_3 < \cdots  < k_i < \cdots$.
Write  $J_1 = P_{k_1}, J_2 = P_{k_2+1}, \cdots, J_n = P_{k_n+1}, \cdots$ and  $R = \prod_{i=0}^\infty \mid J_i(z) \mid^2 $. If $\gamma_n$ denotes the constant term of $ \prod_{i=n+1}^\infty J_i$, then it is easy to see that $\gamma_n > c_0^{(k_n)}$
so that $p_n(1 -\gamma_{n+1}) \leq \frac{1}{2^n}$, where $p_n$ is the degree of $ \prod_{i=1}^n J_i$. By the theorem above the Riesz product $R = \prod_{i=1}^\infty \mid J_i(z)\mid ^2$ is equivalent to the Lebesgue measure.
This completes the proof of the corollary.
\end{proof}
Assume that $b$ is positive. Consider $L(z) = \prod_{j=m+1}^n| P_j(z)|^2$. If $d_k(m,n) = d_j$ is the coefficient of $z^k$
in $\prod_{j = m+1}^n P_j$, then for $k >0$, the coefficient of $z^k$ in $L(z)$ is in absolute value
$$\Bigl| \sum_{j \geq k} d_j\overline{d_{j-k}}\Bigr| \leq \Bigl(\sum_{j \geq k}\mid d_j\mid^2\Bigr)^{\frac{1}{2}}\leq (1 - d^2_0)^{\frac{1}{2}}. $$
Under the assumption that $b$ is positive we can make this coefficient (which depends on $m$ and $n$) as small as we please by choosing $m$ large. We have proved:\\
\begin{Th}\label{th4}
If $b$ is positive, the generalized Riesz products $\mu_{n} \egdef \prod_{n+1}^\infty\mid P_i\mid^2$,
$n=1,2,\cdots$ converge weakly to the Lebesgue measure on $S^1$ as $n \rightarrow \infty$.
\end{Th}
We do not know if the conclusion of Theorem \ref{th4} always holds when $b$ is zero, but such generalized Riesz products form an important class of measures and will be discussed in section 5.

\begin{rem}\label{rem1}\textnormal{The weak dichotomy theorem (Theorem \ref{th1}) is rather weak in the sense that no information can be garnered about $\mu$, such as absolute continuity or singularity, when $b$ is zero. Consider the classical Riesz product
 $$\mu = \prod_{j=1}^\infty\mid\cos (\theta_j) + \sin (\theta_j) z^{n_j}\mid^2,~~~ \frac{n_{j+1}}{n_j} \geq 3,~~ 0 < \theta_j < \frac{\pi}{2},~~ j =1,2,\cdots$$
It is known to be absolutely continuous if $ \sum_{j=1}^\infty \cos^2(\theta_j)\sin^2(\theta_j)$ is finite and singular otherwise. Clearly the condition for absolute continuity is satisfied with $\cos (\theta_j) =  \frac{1}{j}, j=1,2,\cdots$ and also with $\cos (\theta_j) = \sqrt{1 -\left(\frac{1}{j}\right)^2},~j =1,2,\cdots$. In the first case the product of the constant terms is zero, while in the second case it is positive. This defect is rectified if we replace the polynomials $P_j$ with their
outer parts, as discussed in the next section.}
\end{rem}

\section{Outer Polynomials and Mahler Measure.}
 Let
   $$\mu = \prod_{j=1}^\infty \mid P_j(z)\mid^2 \ ~~\eqno (1)$$ be a generalized Riesz product. Let $\mu_a$ denote the part of $\mu$ absolutely continuous with respect to $dz$. We write $ \frac{d\mu}{dz}$, to mean $ \frac{d\mu_a}{dz}$.
In this section we use the classical prediction theoretic ideas to evaluate $\displaystyle \exp{\Bigl({\int_{S^1}\log\Bigl( \frac{d\mu}{dz}\Bigr)dz}\Bigr)}$ a quantity which we call the Mahler measure of $\mu$ (denoted by $M(\mu)$) with respect to the Lebesgue measure.\\

\noindent{}We will prove:
\begin{Th}\label{th5}
\begin{eqnarray*}
\int_{S^1}\log\frac{d\mu_a}{dz}dz &=&\lim_{n\rightarrow \infty}\int_{S^1}\log \prod_{j=1}^n \mid P_j(z)\mid^2dz.
\end{eqnarray*}
\end{Th}
 Note that the theorem is false if we drop the log on both sides of the equation, for then the righthand side is always one, while the left hand side is zero for $\mu$ singular to Lebesgue measure. For the proof we begin by recalling Beurling's inner and outer factors for the case of polynomials and the expression for one step`prediction error', namely the quantity:
$$\inf_{q \in\mathcal Q}\int_{S^1}\mid 1 - q(z)\mid^2 \mid P(z)\mid^2dz,$$ where $P(z)$ is an analytic trigonometric polynomial with $L^2(S^1,dz)$ norm 1 and non-zero constant term.  $\mathcal Q$ is the class of all analytic trigonometric polynomials with zero constant term. To this end we have to bring into consideration the zeros of polynomials $P_j, j=1,2,\cdots$. Consider the $k$th polynomial of the generalized Riesz product product $ \prod_{k=1}^\infty|P_k|^2$. Suppressing the index $k$, it is of the type:
     $$P(z) = a_0 + a_1 z + \cdots + a_{m}z^{m}.$$
assuming that it is of degree $m$.
     Let $$A = \{a: P(a) = 0, \mid a\mid < 1\},$$
     $$B = \{b: P(b) = 0, \mid b\mid = 1\},$$
     $$C = \{c: P(c) = 0, \mid c\mid > 1\}.$$
  Then
  $$P(z) = a_{m} \prod_{a\in A}(z - a)\prod _{b\in B}(z-b) \prod_{c\in C}(z-c)$$
  $$=\prod_{a\in A}\frac{(z -a)}{(1-\overline a z)} a_{m}\prod_{a\in A}(1 - \overline a z) \prod_{b\in B}(z-b)\prod_{c\in C}(z-c).$$
  Write $$I(z) =\overline{\gamma} \prod_{a\in A}\frac{(z -a)}{(1-\overline a z)},$$
  $$O(z) = \gamma a_{m}\prod_{a\in A}(1 - \overline a z) \prod_{b\in B}(z-b)\prod_{c\in C}(z-c).$$
  where $\gamma$ is a constant of absolute value 1 such that the  constant term of $O(z)$ is positive, while $\overline {\gamma}$ is the complex conjugate of $\gamma$.
  We have,

  $$P(z) =  I(z) O(z).$$
  Note that for  $ z\in S^1$, $\mid I(z)\mid =1$, $ \mid P(z)\mid = \mid O(z)\mid $. The function $O(z)$ is non-vanishing inside the unit disc. The functions $I$ and $O$ are Beurling's inner and outer parts of the polynomial $P$. Note that, since constant term of $P$ is non-zero, the  degree of $O$  is same as that of $P$ and that $O(0)$ = constant term of $O$ $\geq P(0) = a_0$. Recall that  outer functions in $H^2$ are precisely those functions $f$ in $H^2$ for which the functions $z^n f, n \geq 0$ span $H^2$ in the closed linear sense. Hence, if $f$ is an outer function in $H^2$, then the closed linear span of $\{z^nf,n\geq 1\}$  is $ zH^2$. The orthogonal projection of $O(z)$ on $zH^2$ is $O(z) - O(0)$ where  $O(0)$ is the constant term of $O(z)$ which we denote by $\alpha$. Note that
   $$\bigl| \alpha \bigr|= \Bigl| a_{m}\prod_{b\in B}b\prod_{c\in C}c\Bigr|=|a_m|\prod_{c\in C}|c|.$$

We have

$$\bigr| \alpha\bigl|^2 = \bigintss_{S^1}\Bigl|\alpha\mid^2dz = \bigintss_{S^1}\Bigl| \bigl(O(z) -(O(z) -\alpha\bigr)\Bigr|^2dz$$
$$= \bigintss_{S^1}\Bigl| 1 - \frac{\bigl(O(z) - \alpha\bigr)}{O(z)}\Bigr|^2 \Bigl| O(z)\Bigr|^2dz$$
$$= \inf_{q\in{\mathcal Q}} \bigintss_{S^1}\Bigl| 1 - q(z)\Bigr|^2 \Bigl| O(z)\Bigr|^2dz,$$
$$= \inf_{q\in{\mathcal Q}} \bigintss_{S^1}\Bigl| 1 - q(z)\Bigr|^2 \Bigl| P(z)\Bigr|^2dz$$

where the infimum is taken over the class $\mathcal Q$ of all analytic trigonometric polynomials $q$ with zero constant term.
Thus $\frac{O(z) - \alpha}{O(z)}$ is the orthogonal projection of the constant function $1$ on the closed linear span of $\{z^n, n \geq 1\}$ in $L^2(S^1,| P(z)|^2dz)$.\\

\begin{lem}\label{lem1}
 If $\lambda$ is a probability measure on $S^1$ such that $d\nu=\mid P(z)\mid^2d\lambda$ is again a probability measure then
$$\mid \alpha \mid^2 \geq \inf_{q \in {\mathcal{Q}}}\int_{S^1}\mid 1-q(z)\mid^2d\nu.$$
\end{lem}
\begin{proof}
If $O(z)$ has no zeros on the unit circle then $\ds \frac{O(z) - \alpha}{O(z)}$ is analytic on the closed unit disk. The partial sums of the power series of this function converge to it uniformly on the unit circle. Let $q_k, k = 1,2,\cdots$ be the sequence of these partial sums.
Then
$$\mid\alpha\mid^2 = \bigintss_{S^1}\Bigl| 1 - \frac{O(z) -\alpha}{O(z)}\Bigr|^2\Bigl| O(z)\Bigr|^2d\lambda$$
$$ = \bigintss_{S^1}\Bigl|1 - \frac{O(z) -\alpha}{O(z)}\Bigr|^2d\nu$$
$$=\lim_{k\rightarrow \infty}\bigintss_{S^1}\Bigl| 1 - q_k\Bigr|^2d\nu \geq \inf_{q\in {\mathcal Q}}\bigintss_{S^1}\Bigl| 1 - q(z)\Bigr|^2d\nu$$
This conclusion remains valid even if $O(z)$ has zeros on the circle but the proof is slightly different. For fixed
$r$, $0 \leq r < 1$, the function $\frac{O(z) -\alpha}{O(rz)}$ is analytic on the closed unit disk, so the partial sums of its power series  converge to it uniformly on $S^1$. Now for any fixed real $\theta$,  $\ds \frac{z-e^{i\theta}}{rz - e^{i\theta}}$ remains bounded by 2 for $z\in S^1$
and $0\leq r <1$, and converges to 1 as $r\rightarrow 1$, for $z \neq e^{i\theta}$. Therefore $z\neq \theta$,
$\ds \frac{O(z)}{O(rz)} \rightarrow 1$ boundedly as $r\rightarrow 1$, whence
$$\Bigl| 1 - \frac{O(z) -\alpha}{O(rz)}\Bigr|^2\Bigl|O(z)\Bigr|^2 \rightarrow \left|\alpha\right|^2 $$ boundedly as $r\rightarrow 1$. It is easy to see from this that
$$\inf_{q\in  {\mathcal Q}}\bigintss_{S^1}\Bigl| 1 - q\Bigr|^2d\nu \leq \lim_{r\rightarrow 1}\bigintss_{S^1}\Bigl| 1 - \frac{O(z) - \alpha}{O(rz)}\Bigr|^2 d\nu = \mid \alpha\mid^2.$$
This proves the lemma.
\end{proof}

Consider now the polynomials $P_k,$ $k =1,2,\cdots$  and the associated finite products
$\prod_{k=1}^{n}P_k,$ $n =1,2,3,\cdots$. Let $A_k, B_k, C_k, I_k, O_k, \alpha_0^{(k)}$ have the obvious meaning: they are for $P_k$ what $A, B, C, I, O, \alpha$ are for  $ P$. Note that the inner and outer parts of $\prod_{k=1}^nP_k$ are $\prod_{k=1}^nI_k$ and $\prod_{k=1}^nO_k$
respectively and the constant term of the outer part is $\prod_{k=1}^n\alpha_0^{(k)}$. Note that $ \prod_{k=1}^\infty \alpha_0^{(k)} \egdef \beta \geq \prod_{k=1}^\infty a_0^{(k)} = b$, so if $b$ is positive, then so is $\beta$.
On the other hand, the positivity of $\beta$ does not in general imply positivity of $b$ as shown by the case of classical Riesz product (see remark \ref{rem2} below).\\

To prove Theorem \ref{th5} We apply the lemma  above to $$P = \prod_{k=1}^n \emph{}P_k(z)$$ and
$$\lambda = \prod_{k=n+1}^\infty\mid P_k(z)\mid^2,$$
Note that $$d\nu =\Bigl (\prod_{k=1}^n\mid P_k(z)\mid^2\Bigr)d\lambda = d\mu.$$
We see that for any $n$
$$\inf_{q\in {\mathcal Q}} \bigintss_{S^1}\Bigl| 1 - q\Bigr|^2 d\mu \leq \prod_{k=1}^n\mid\alpha_0^{(k)}\mid^2,$$
whence
  $$\inf_{q\in {\mathcal Q}} \bigintss_{S^1}\Bigl| 1 - q\Bigr|^2 d\mu \leq \prod_{k=1}^\infty\mid\alpha_0^{(k)}\mid^2.$$
\noindent{}To prove the above inequality in the reverse direction we note that by
Szeg\"o's theorem as generalized by Kolmogorov and Krein (K. H. Hoffman \cite{Hoffman}) that
$$\exp \Biggr\{\bigintss_{S^1}\log \Bigr(\frac{d\mu_a}{dz}\Bigl)dz\Biggl\} = \inf_{q\in {\mathcal Q}} \bigintss_{S^1}\mid 1 - q(z)\mid^2 d\mu.$$ Denote this infimum by $l$. Then, given $\epsilon > 0$, there is a polynomial $q_0\in {\mathcal Q}$ such that
$$l \leq \bigintss_{S^1}\mid 1 - q_0\mid^2d\mu < l + \epsilon,$$
\noindent{}whence for large enough $n$

$$ \bigintss_{S^1}\Bigl| 1 - q_0\Bigr|^2\prod_{k=1}^n\left| P_k\right|^2dz < l +\epsilon.$$
\noindent{}Since $$\Bigl|\prod_{k=1}^n \alpha_0^{(k)}\mid^2 =\inf_{q\in {\mathcal Q}}\bigintss_{S^1}\mid 1- q\mid^2 \prod_{k=1}^n\mid P_k(z)\mid^2dz,$$
\noindent{}we see that $\mid \alpha_0^{(1)}\cdot\alpha_0^{(2)}\cdots \alpha_0^{(n)}\mid^2 \leq l +\epsilon$. Since $\epsilon $ is arbitrary positive real number, and $\mid \alpha_0^{(k)} \mid <  1$ for all $k$, we have
$$\prod_{k=1}^\infty\mid\alpha_0^{(k)}\mid^2 \leq l.$$

\noindent{}We also note that
$$\prod_{j=1}^n\mid\alpha_0^{(j)}\mid = \exp\Biggl\{\bigintss_{S^1}\log \bigl(\mid \prod_{j=1}^n P_j(z)\mid\bigr) dz\Biggr\}$$

\noindent{}Thus we have proved (see \cite{Nadkarni2}):
 \begin{eqnarray*}
\prod_{k=1}^\infty\mid\alpha_0^{(k)}\mid^2 &=& \exp \Biggl\{\bigintss_{S^1}\log \Bigr(\frac{d\mu_a}{dz}\Bigl)dz\Biggr\}\\
&=&\lim_{n\rightarrow \infty}\exp\Biggl\{\bigintss_{S^1}\log \bigl(\prod_{j=1}^n \mid P_j(z)\mid^2 \bigr)dz\Biggr\}.
\end{eqnarray*}
which is indeed theorem \ref{th5} with some additional information.\\

\begin{Cor}
If each $P_i, i=1,2,\cdots$ is outer, then $\log \Bigl(\frac{d\mu}{dz}\Bigl)$ has finite integral if and only if $\beta$ is positive.
\end{Cor}

\begin{rem}\label{rem2}\begin{enumerate}[i)]
\item \textnormal{For the trigonometric polynomial $P(z) = \cos \theta_j + \sin \theta_j z^{n_j}$ appearing in the classical Riesz product of  remark \ref{rem1}, its outer part has the constant term $ \max \{\cos \theta_j, \sin \theta_j\}$ and the condition
$\sum_{j=1}^\infty\cos^2\theta_j\sin^2\theta_j <\infty$ is equivalent to the condition
$ \prod_{j=1}^\infty \max \{\cos \theta_j, \sin \theta_j\}$ is positive. The additional information we have now is that in case $\mu$ is absolutely continuous with respect to $dz$, $\log\frac{d\mu}{dz}$ has finite integral.}

 \item \textnormal{ Using a deep result of Bourgain \cite{bourgain}, the first author has shown recently \cite{elabdal} that if the cutting parameter $m_n, n=1,2,\cdots$ of a rank one transformation $T$ satisfies
$\frac{m_n}{n^\beta} < K$ for all $n$ for some constant $K$ and some $\beta \in (0,1]$, then the Mahler measure of the spectrum of $T$ is zero.}
\item \textnormal{It is to be noted that if each $P_n$ is outer and the product
$\prod_{k=1}^\infty P_k(0)$ is non-zero, then formal expansion $f$ of $\prod_{k=1}^\infty P_k(z)$ is an outer function. Indeed the  Mahler measure of $\mid f\mid^2$ is $\mid f(0)\mid^2$, so $f$ can not admit a non-trivial inner factor. Also the measure $1\cdot dz$ can be expressed as a generalized Riesz product only by choosing each $P_n =1$, for if any of the $P_n$ is not the constant equal to 1, then its normalized outer part will have constant term less than one, which will force the Mahler measure of $1\cdot dz$ to be less than 1, which is false. It is not known if the measure $cdz + d\delta_{1}, c, d > 0, c+d =1$ can be expressed as a generalized Riesz product, where $\delta_{1}$ denotes the Dirac measure at
1.}
\item \textnormal{ Let $\mu$ be a probability measure on $S^1$, and let $q$ be a natural number. We contract the measure to the arc $A = \{z: z = \exp\{i\theta\}, 0\leq \theta <\frac{2\pi}{q}\}$, namely we consider the measure $\nu_1$ supported on this arc given by $\nu_1(B) = \mu ({z^q: z \in B}), B \subset A$. We write similarly $\nu_i(C)  = \nu_1(\exp\{-\frac{2\pi i}q\}C), C \subset \exp\{\frac{2\pi i}{q}\}A.$ Let
    $\Pi_q(\mu)= \frac{1}{q}\sum_{i=1}^q\nu_i.$
  It can be verified that if $\mu =\mid p(z)\mid^2dz$, then $$\Pi_q(\mu) = \mid P(z^q)\mid^2dz,$$ from which we conclude that
  if $\mu = \Pi_{k=1}^\infty\mid P_k(z)\mid^2$ then
  $$\Pi_q(\mu) = \Pi_{k=1}^\infty\mid P_k(z^q)\mid^2.$$
We see immediately that the Mahler measure of a generalized Riesz product is
invariant under the application of $\Pi_q$ for any $q$.}

\end{enumerate}
\end{rem}

\section{A Formula for Radon Nikodym Derivative.}

Consider two generalized Riesz products $\mu$ and $\nu $ based on polynomials $P_j, j=1,2,\cdots$ and $Q_j, j =1,2,\cdots$ where $\nu$ is continuous except for a possible mass at $1$. Under suitable assumptions we prove the formula:\\
$$  \frac{d\mu}{d\nu} =\lim_{n\rightarrow \infty}\frac{\prod_{j=1}^n \mid P_j\mid^2}{\prod_{j=1}^n\mid Q_j\mid^2},$$
in the sense of $L^1(S^1, \nu)$ convergence.\\
Let $\sigma$ and $\tau$ be two measures on the circle. Then, by Lebesgue decomposition of  $\sigma$ with respect to $\tau$, we have
\[
\sigma=\frac{d\sigma}{d\tau} d\tau+\sigma_s,
\]
where $\sigma_s$ is singular to $\tau$ and $\frac{d\sigma}{d\tau}$ is the Radon-Nikodym derivative. In the case of two Riesz products $ \mu = \prod_{j=1}^\infty\mid P_j\mid^2$ and $ \nu =\prod_{j=1}^\infty\mid Q_j\mid^2$, we are able to see  that  their affinities, namely the ratios $ \frac{\prod_{j=1}^n\mid P_j\mid}{\prod_{j=1}^n\mid Q_j\mid}, k=1,2,\cdots$,  converge in $L^1$ to $\sqrt\frac{d\mu}{d\nu}$, assuming that $\nu$ has no point masses except possibly at $1$. This result extends a theorem of G. Brown and W. Moran \cite{Brown-Moran}. Let $\delta_1$ denote the unit mass at one. We have ( see \cite{elabdal})

\begin{Th}\label{th6}
Let $ \mu=\prod_{j=0}^{\infty}\mid P_j\mid^2$,
  $ \nu=\prod_{j=0}^{\infty}\mid Q_j\mid^2$ be two generalized Riesz products. Let
  $$\mu_n=\prod_{j=n+1}^{\infty}\mid P_j\mid^2,~~
   \nu_n=\prod_{j=n+1}^{\infty}\mid Q_j\mid^2
  $$
Assume that
\begin{enumerate}
  \item  $\nu=\nu'+b\delta_1$, $\nu'$ is continuous measure, $0\leq b <1$.
  \item $ \prod_{j=0}^{n}\mid P_j\mid^2 d\nu_n \longrightarrow \mu$ weakly as $n\longrightarrow \infty$
  \item $ \prod_{j=0}^{n}\mid Q_j\mid^2 d\mu_n \longrightarrow \nu$ weakly as $n\longrightarrow \infty$
\end{enumerate}
Then the finite products $ R_n= \prod_{k=1}^{n}{\left|\frac{P_k(z)} {Q_k(z)}\right|}, n=1,2,\cdots$ converge in $L^1(S^1,\nu)$ to $\sqrt{\frac{d\mu}{d\nu}}$.
\end{Th}

\noindent{} To prove this we need the following proposition.
\begin{Prop}\label{prop1}
The sequence $\ds  \prod_{j=0}^{n}\left|\frac{P_j(z)}{Q_j(z) }\right|, n=1,2,\cdots$ converges weakly in $L^2(S^1,\nu)$ to $\ds \sqrt{\frac{d\mu}{d\nu}}$.
\end{Prop}

\begin{proof} Put
$ f=\sqrt{\frac{d\mu}{d\nu}}$
and let $n$ be a positive integer. Now

$$\bigintss_{S^1} R_n^2 d\nu = \bigintss_{S^1} \prod_{j=1}^{n}\mid {}P_j|^2 d\nu_n \rightarrow \bigintss_{S^1}d\mu = 1$$
by assumption (2). Hence $\int_{S^1} R_n^2 d\nu , n =1,2,\cdots$ remain bounded. Thus, the weak closure of $ R_n(z), n =1,2,\cdots$ in $ L^2(S^1, \nu)$ is not empty.

 We show that this weak closure has only one point, namely,
$\sqrt{\frac{d\mu}{d\nu}}$. Indeed, let $g$ be a weak subsequential limit, say, of $R_{n_j}(z), j =1,2,\cdots$. Then, for any continuous positive function $h$, we have, by judicious applications of Cauchy-Schwarz inequality,

$$\Biggl(\bigintss_{S^1} f(z) h(z) d\nu(z)\Biggr)^2 = \Biggl(\bigintss_{S^1} h(z) R_{n_j}(z) \frac{1}{R_{n_j}}\sqrt{\frac{d\mu}{d\nu}}  d\nu(z)\Biggr)^2 $$
 $$\leq  \Biggl(\bigintss_{S^1} h(z) R_{n_j}(z) d\nu(z)\Biggr) \Biggl(\bigintss_{S_1} h(z) R_{n_j}(z) \frac{1}{R_{n_j}^2(z)}\frac{d\mu}{d\nu} d\nu(z)\Biggr)$$
$$\leq \Biggl(\bigintss_{S^1} h(z) R_{n_j}(z) d\nu(z)\Biggr) \Biggl(\bigintss_{S^1} h(z) \frac{1}{R_{n_j}(z)} d\mu\Biggr)$$
 $$\leq \bigintss_{S^1} h(z) R_{n_j}(z) d\nu(z) \Biggl(\bigintss_{S^1} h(z) d\mu \Biggr)^{\frac{1}{2}}\Biggl(\bigintss_{S^1} h(z) \frac{d\mu}{R_{n_j}^2(z)}\Biggr)^{\frac{1}{2}}$$
$$\leq  \Biggl(\bigintss_{S^1} h(z) R_{n_j}(z) d\nu(z)\Biggr) \Biggl(\bigintss_{S^1} h(z) d\mu\Biggl)^{\frac{1}{2}}
\Biggl(\bigintss_{S^1} h(z)\mid\prod_{k=1}^{n_j} Q_k\mid^2 d\mu_{n_j}\Biggr)^{\frac{1}{2}} $$

Letting $j \rightarrow +\infty$, from our assumption (3), we get

$$\Biggl(\bigintss_{S^1} f h d\nu\Biggr)^2 \leq
\Biggl(\bigintss_{S^1} h g d\nu\Biggr) \Biggl(\bigintss_{S^1} h d\mu\Biggr)^{\frac{1}{2}}
\Biggl(\bigintss_{S^1} h d\nu\Biggr)^{\frac{1}{2}}  \eqno (2).$$

 But, since the space of continuous functions is dense in $L^2(\mu+\nu)$, we deduce from (2) that, for any Borel set $B$,

$$\Biggl(\bigintss_{B} f  d\nu\Biggr)^2 \leq
\Biggl(\bigintss_{B}  g d\nu\Biggr) \Biggl(\bigintss_{B} d\mu\Biggr)^{\frac{1}{2}}
\Biggl(\bigintss_{B} d\nu\Biggr)^{\frac{1}{2}}.
$$
By taking a Borel set $E$ such that $\mu_s(E)=0$ and $\nu(E)=1$, we thus get, for any $B \subset E$,

$$\Biggl(\bigintss_{B} f  d\nu\Biggr)^2 \leq
\Biggl(\bigintss_{B}  g d\nu\Biggr) \Biggl(\bigintss_{B} f^2(z) d\nu \Biggr)^{\frac{1}{2}}
\Biggl(\bigintss_{B} d\nu\Biggr)^{\frac{1}{2}}.$$

It follows from Martingale convergence theorem that:
$$ f(z) \leq g(z) {\textrm {~for~almost~all~}}z {\textrm{~with~respect~to~}} \nu.$$

Indeed, let ${{\mathcal  P}}_n = \{A_{n,1}, A_{n,2}\cdots, A_{n, k_n}\},$  $n =1,2,\cdots$, be a refining sequence of finite partitions of $E$ into Borel sets such that such that they tend to the partition of singletons. If $\{x\} = \cap_{n=1}^\infty A_{n,j_n}$,

$$\Biggl(\frac{1}{\mu (A_{n, j_n})}\bigintss_{B} f  d\nu\Biggr)^2 \leq$$
$$
\Biggl(\frac{1}{\mu (A_{n, j_n})}\bigintss_{A_{n,j_n}}  g d\nu\Biggr) \Biggl(\frac{1}{\mu (A_{n, j_n})}\bigintss_{A_{n,j_n}} f^2(z) d\nu \Biggr)^{\frac{1}{2}}
\Biggl(\frac{1}{\mu (A_{n, j_n})}\bigintss_{A_{n,j_n}} d\nu\Biggr)^{\frac{1}{2}}.$$
Letting $n\rightarrow \infty$ we have, by Martingale convergence theorem as applied to the theory of derivatives,
for a.e $x \in E$ w.r.t. $\nu$,
$$(f(x))^2 \leq g(x) f(x), ~~{\rm{whence}} ~~f(x)\leq g(x)$$

\noindent{}For the converse note that for any continuous positive function $h$ we have
$$\bigintss_{S^1} g h d\nu = \lim_{j \longrightarrow +\infty}\bigintss_{S^1} h(z) R_{n_j}   d\nu $$
$$ \leq   \lim_{j \longrightarrow \infty} \Biggl(\bigintss h R_{n_j}^2 d\nu\Biggr)^{\frac{1}{2}} \Biggl(\bigintss_{S^1} h d\nu\Biggr)^{\frac{1}{2}}$$
$$ \leq  \Biggl(\bigintss_{S^1} h d\mu\Biggr)^{\frac{1}{2}}  \Biggl(\bigintss h  d\nu\Biggr)^{\frac{1}{2}}.$$

\noindent{}As before we deduce that $g(z) \leq f(z)$ for almost all $z$ with respect to $\nu$.
Consequently, we have proved that $g=f$ for almost all $z$ with respect to $\nu$ and this complete the proof of the proposition.
\end{proof}
\begin{proof}[\bf {Proof of Theorem \ref{th6}}] We will show that
    $\beta_n \stackrel{\textrm {def}}{=} \bigintss_{S^1} \mid R_n-f\mid d\nu \rightarrow 0$ as $n \rightarrow \infty$,
where $f=\sqrt{\frac{d\mu}{d\nu}}$. Now,

$$\frac{d\mu}{d\nu}=R_n^2(z)\frac{d\mu_n}{d\nu_n}~~ {\rm{and}}~~  \sqrt{\frac{d\mu}{d\nu}}=R_n(z)\sqrt{\frac{d\mu_n}{d\nu_n}}$$

\noindent{}Put $$f^2_n =\frac{d\mu_n}{d\nu_n},$$ Then,

 $$\bigintss_{S^1}f_n^2d\nu  = \bigintss_{S^1}\prod_{k=1}^n\mid Q_k\mid^2d\mu_n \rightarrow \bigintss_{S^1}d\nu =1,$$
by assumption (3). The functions $f_n, n=1,2,\cdots$ are therefore bounded in $L^2(S^1, \nu)$.
 Hence, there exists a subsequence
$f_{n_j} = \sqrt{\frac{d\mu_{n_j}}{d\nu_{n_j}}}, j = 1,2, \cdots$ which converges weakly to some $L^2(S^1, \nu)$-function $\phi$. We show that $0 \leq \phi \leq 1$ a.e ($\nu$). For any continuous positive function $h$, we have

$$\Biggl(\bigintss_{S^1} h f_{n_j} d\nu\Biggr)^2 \leq \Biggl(\bigintss_{S^1} h d\nu\Biggr) \Biggl(\bigintss_{S^1} h f_{n_j}^2 d\nu\Biggr)$$
$$\leq \Biggl(\bigintss_{S^1} h d\nu\Biggr) \Biggl(\bigintss_{S^1} h \frac{d\mu_{n_j}}{d\nu_{n_j}}d\nu\Biggr).$$

\noindent{} Hence, by letting $j$ go to infinity combined with our assumption (3), we deduce that
$$ \bigintss_{S^1} h(z) \phi(z) d\nu \leq  \bigintss_{S^1} h(z) d\nu.$$
Since this hold for all continuous positive functions $h$, we conclude that $0 \leq \phi \leq 1$ for almost all $z$ with respect to $\nu$. Thus any subsequential limit of the sequence $f_n, n=1,2,\cdots$ assumes values between $0$ and $1$.
Now, for any subsequence $n_j, j =1,2,\cdots$ over which $f_{n_j}, j=1,2,\cdots$ has a weak limit , from our assumption (2) combined with Cauchy-Schwarz inequality, we have
$$
 \Biggl(\bigintss_{S^1}|R_{n_j}-f| d\nu\Biggr)^2=\Biggl(\bigintss_{S^1} |R_{n_j}-R_{n_j} f_{n_j}| d\nu\Biggr)^2$$

$$=\Biggl(\bigintss_{S^1} R_{n_j}|1-f_{n_j}| d\nu\Biggr)^2 $$
$$\leq \Biggl(\bigintss_{S^1} R_{n_j}|1-f_{n_j}|^2 d\nu\Biggr) \Biggl(\bigintss_{S^1} R_{n_j} d\nu \Biggr)$$
$$\leq \Biggl(\bigintss_{S^1} R_{n_j} d\nu-2\bigintss_{S^1} R_{n_j}f_{n_j}d\nu+\bigintss_{S^1} R_{n_j}(f_{n_j})^2 d\nu\Biggr)\Biggl(\bigintss_{S^1}R_{n_j}d\nu\Biggr)$$

$$\leq \Biggl(\bigintss_{S^1} R_{n_j} d\nu-2 \bigintss_{S^{1}} f d\nu+\bigintss_{S^1} R_{n_j}f_{n_j}. f_{n_j}  d\nu\Biggr)\Biggl(\bigintss_{S^1}R_{n_j}d\nu\Biggr)$$

$$\leq \Bigl(\bigintss_{S^1} R_{n_j} d\nu-2 \bigintss_{S^1} f d\nu+\bigintss_{S^{1}} f. f_{n_j}  d\nu\Bigr)\Bigl(\bigintss_{S^1}R_{n_j}d\nu\Bigr)$$

\noindent{}Hence, letting $j$ go to infinity,
$$\Biggl(\lim_{j \rightarrow \infty}\bigintss_{S^1}\mid R_{n_j}-f\mid d\nu\Biggr)^2$$
$$\leq \bigintss_{S^1} f d\nu-2 \bigintss_{S^1} f d\nu+\bigintss_{S^1} f. \phi  d\nu$$
$$\leq \bigintss_{S^1}(\phi(z)-1)f(z) d\nu(z).$$
  $$\leq 0,$$
\noindent{}and this implies that $R_{n_j}, j =1,2,\cdots $  converges to $ f$ in $L^1(S^1, \nu)$  and the proof of the theorem is achieved.
\end{proof}
\begin{rem}\label{rem3}
\textnormal{Notice that
$\ds \int_{S^1} \frac{d\mu}{d\nu} d\nu=1,$
implies the convergence of $ \prod_{j=0}^{N}|R_j|$ to $ \sqrt{\frac{d\mu}{dz}}$ in $L^2(d\nu)$, by virtue of the classical results on ``when weak convergence implies strong convergence".}
\end{rem}
\noindent{}We further have \cite{Nadkarni1}
\begin{Cor}\label{cor3}
Two generalized Riesz products $\mu = \prod_{j=1}^\infty\big| P_j\big|^2$,
$\nu = \prod_{j=1}^\infty\big| Q_j\big|^2$  satisfying the conditions of Theorem \ref{th6} are mutually singular
if and only if $$\bigintss_{S^1} \prod_{j=0}^n\Big|\frac{P_j}{Q_j}\Big| d\nu \rightarrow 0~~ {\rm{as}} ~~n\rightarrow \infty.$$
\end{Cor}

\section{A Conditional Strong Dichotomy and Other Discussion.}

An important class of generalized Riesz products is the one arising in the study of rank one transformations of ergodic theory \cite{Nadkarni1}. Indeed much of the recent work on generalized Riesz products (including the present contribution) is motivated by or focussed on the question whether these
generalized Riesz products are always singular to Lebesgue measure. For in the contrary case, a counter example to this belief, will in all probability yield an affirmative answer to an old problem of Banach as to whether there exists a measure preserving transformation on an atom free measure space with simple Lebesgue spectrum.

The $k^{th}$ polynomial in the generalized Riesz product arising in the study of measure preserving rank one transformation is
 of the type
 $$P_k(z) = \frac{1}{\sqrt{m_k}}(1 + \sum_{j=1}^{m_k-1}z^{jh_{k-1} + a_1^{(k)} + \cdots  + a_j^{(k)}}),$$
 $$h_{k-1}={m_{k-2}h_{k-2}}+\sum_{j=1}^{m_{k-2}}a_j^{(k)}, h_0=1,$$
\noindent{}where $m_k$'s are the cutting parameter and $a_j^{(k)}$'s are the spacers of the rank one transformation under consideration. It is easy to see that the partial products $\prod_{j=1}^k P_j, j =1,2,\cdots$ converge weakly to zero in $L^2(S^1, dz)$. These generalized Riesz products $\prod_{j=1}^\infty| P_j|^2$ have the property that
 the sequence of their tails $\prod_{j=n+1}^\infty| P_j|^2,$ $n=1,2,\cdots$ converges weakly to the Lebesgue measure. In the rest of this section we will assume that the generalized Riesz products have this additional property,
 although  is not assumed that they arise from rank one transformations as above.
\begin{Def}\label{def3}
 A generalized Riesz product $\mu  = \prod_{j=1}^\infty |P_j|^2 $ is said to be of class
 (L) if for each sequence $k_1 < k_2 < \cdots$ of natural numbers, the tail measures
 $\prod_{j=n+1}^\infty\mid{P_{k_j}} \mid^2, n = 1,2,\cdots$ converge weakly to  Lebesgue measure.
\end{Def}
\begin{Prop}\label{prop2}
If the generalized Riesz product $\mu =\prod_{j=1}^\infty\mid P_j\mid^2$ is of class (L) then the partial products $\prod_{j=1}^n\mid P_j\mid, n=1,2,\cdots$
converge in $L^1(S^1, dz)$ to $\sqrt{\frac{d\mu}{dz}}$, and the convergence is almost everywhere (w.r.t $dz$) over a subsequence.
\end{Prop}
\begin{proof}
In Theorem \ref{th6}  we put $Q_j(z) =1$ for all $j$, so that $\nu$ is the Lebesgue measure on $S^1$. The first conclusion follows from  theorem \ref{th6}. The second conclusion follows since $L^1$ convergence implies convergence a.e over a subsequence.
\end{proof}
\noindent{}The following formula follows immediately from this:
\begin{Cor}\label{cor4}
Let a generalized Riesz product $\mu$ be of class (L). Let ${{\mathcal K} }_1,{{\mathcal K}}_2$
be two disjoint subsets of natural numbers and let ${\mathcal K}_0$ be their union.
Let $\mu_1, \mu_2$ and $\mu_0$ be the generalized Riesz subproducts of $\mu$ over
${\mathcal K}_1, {\mathcal K}_2$, and ${\mathcal K}_0$  respectively. Then we have:
$$\frac{d\mu_0}{dz} = \frac{d\mu_1}{dz}\frac{d\mu_2}{dz},  \eqno (1)$$
where equality is a.e. with respect to the measure $dz$.
\end{Cor}

Let  $\mu = \prod_{j=1}^\infty \mid P_j\mid^2$ be a generalized Riesz product of class (L). We assume that the polynomials $P_j, j=1,2,\cdots$ are outer. Write
 $S_n = \prod_{j=1}^nP_j$, and let  $\phi_n = \frac{S_n}{|S_n|}$, a function of absolute value one.
 The functions $\phi_n, n=1,2,\cdots$ admit weak$*$ limits as functions in $L^{\infty}(S^1, dz)$. By Theorem \ref{th6} if $\beta$ is positive then there is a unique nowhere vanishing weak star limit $\frac{f}{| f|}$ which is indeed also a limit in $L^1(S^1, dz)$. On the other hand consider the simplest classical Riesz product
given by
$$\mu = \prod_{j=1}^\infty  \frac{1}{\sqrt 2}\Big| 1 + z^{n_j}\Big|^2,~~~ \frac{n_j}{n_{j-1}} \geq 3,$$ which is singular to the Lebesgue measure on $S^1$.
Since $ 1 + e^{it}={\mid 1 + e^{it}\mid} e^{i\frac{t}{2}}$, we see that
$$\phi_k (e^{it})=e^{i\left(\sum_{j=1}^k n_j\right)\frac{t}{2}} \longrightarrow 0$$
in the weak$*$ topology as $k \rightarrow \infty$, by virtue of the Riemann-Lebesgue lemma. However the following conditional strong dichotomy holds.
\begin{Th}\label{th7}
If the functions $\phi_n, n =1,2,\cdots$ admit a weak star limit $\phi$ in $L^{\infty}(S^1, dz)$  which is non-vanishing a.e. (dz) on the set $\{z: \frac{d\mu}{dz} > 0\}$, then  $\mu$ is either singular to Lebesgue measure, or its absolutely continuous part has positive Mahler measure.
\end{Th}
\begin{proof}
Let $f = \sqrt\frac{d\mu}{dz}$. Fix an integer $k$, then
$$\Biggl|\bigintss_{S^1}z^kS_n(z) - z^k\phi_n(z)f(z)dz\Biggr| = \Biggl|\bigintss_{S^1}z^k\phi_n(z)(\left| S_n(z)\right| - f(z))dz\Biggr|$$
$$\leq \bigintss_{S^1}\left| \left|S_n(z)\right| - f(z)\right| dz \rightarrow 0 ~~{\rm{as}}~~ n \rightarrow \infty$$
by Theorem \ref{th6}.
On the other hand, by assumption, since $f\in L^1(S^1,dz)$,
$$\bigintss_{S^1}z^k\phi_n(z)f(z)dz \rightarrow \bigintss_{S^1}z^k\phi (z) f(z)dz$$

Now for $k < 0, \int_{S^1}z^k S_n(z)dz =0$, so for $k < 0$,
$$\bigintss_{S^1}z^k\phi (z)f(z)dz =0$$

By F and M Riesz theorem $\phi f$ is either the zero function or a non-zero function in $H^1$. In the first case
$f$ is the zero function a.e $dz$, since $\phi$ is assumed to be non-vanishing a.e. $(dz)$ on the set where $f$ is positive. In the second case
 $\mid\phi f\mid$ has an integrable log, which implies that $f$ has an integrable log. Thus $f$ is either the zero function or has  an integrable log.
\end{proof}
\begin{rem}\label{rem4}\textnormal
The proof of Theorem \ref{th7} in fact shows that any weak limit $\phi$ of $\phi_n$'s  either never vanishes or vanishes on the set where $f$ does not vanish. Suppose $S_n$ has degree $m_n$ and let $z_1, z_2, \cdots, z_{m_n}$ be the zeros of $S_n$, counting multiplicity. Since $S_n$ is outer, $\mid z_j\mid \geq 1, j=1,2,\cdots, m_n$, whence
$\mid \frac{1}{z_j}\mid \leq 1, j =1,2,\cdots, m_n$, so the function $(1 - \frac{z}{z_j}), j =1,2,\cdots, m_n$ has  continuous arguments except when $z_j$ has absolute value 1, in which case $z_j$ is the only point where the
argument is not defined, and a continuous argument can be defined at all other points. Thus the polynomials $ B_n \egdef\prod_{j=1}^{m_n}(1 - \frac{z}{z_j})$ admits a continuous argument, denoted by $A_n$,  except at points $z_j$ with $\mid z_j\mid =1$. If $\lim_{n\rightarrow \infty} A_n(z)$ exists at almost every point where $\frac{d\mu}{dz}$
is positive, then it is clear that $\phi_n, n=1,2,\cdots$ admit a weak limit not vanishing a.e. on the set $\{z: \ds \frac{d\mu}{dz} >0\}$. Theorem \ref{th7}  is a soft version in the context of generalized  Riesz product of similar results in the context of
lacunary series (see Theorem 1.1 and  Theorem 6.4 in \cite[T1, p.202]{Zygmund}).).
\end{rem}
   View the functions $S_n(z), n=1,2,\cdots$ as outer analytic functions on the open unit disk.
From  weak dichotomy theorem 1.1, we immediately see that $S_n, n=1,2,\cdots$ converge uniformly
on every compact subset of the open unit disk to a function which is non-zero and in $H^1$ if $\beta$ is positive and the identically zero function if $\beta$ is zero.
We have, using notation from $H^p$ theory, with $0\leq r <1$
$$\lim_{r\rightarrow 1}\frac{1}{2\pi}\bigintss_0^{2\pi}\left| S_n(re^{i\theta})\right| d\theta = \frac{1}{2\pi}\bigintss_0^{2\pi}\left| S_n(e^{i\theta})\right| d\theta$$

\noindent{}We prefer to write this in our notation as

$$\lim_{r\rightarrow 1}\bigintss_{S^1}\left| S_n(rz)\right| dz = \bigintss_{S^1}\left| S_n(z)\right| dz$$

Letting $n\rightarrow \infty$ we get

$$\lim_{n\rightarrow \infty}\Biggl(\lim_{r\rightarrow 1}\bigintss_{S^1}\mid S_n(rz)\mid dz\Biggr) = \lim_{n\rightarrow \infty}\Biggl(\bigintss_{S^1}\mid S_n(z)\mid dz\Biggr)=\bigintss \sqrt{\frac{d\mu}{dz}}~dz.$$
However, in general one can not  interchange the order of taking limits and write this as
$$=\lim_{r\rightarrow 1}\lim_{n\rightarrow \infty}\bigintss_{S^1}\mid S_n(rz)\mid dz,$$
for that would immediately establish the singularity of $\mu$ with respect to the Lebesgue measure when $\beta =0$, which is false in general, see remark 6.7.

\section{Non-Singular Rank One Maps and Flat Polynomials.}

In this section we will discuss generalized Riesz product in connection with spectral questions about non-singular and measure preserving rank one transformations. We will
 give necessary and sufficient conditions under which a generalized Riesz product is the maximal spectral type (up to possibly a discrete component) of a unitary operator associated with a rank one non-singular transformation and certain functions of absolute value one.
  We will expose in more detail and generality, possibly adding new perspective, the known connection of these questions (see \cite{Guenais}, \cite{down}) with problems about flat polynomials. Proposition 5.2 is particularly useful in this discussion.\\

\paragraph{\textbf{Non-Singular Rank One Maps.}}

Let $T$ be a non-singular rank one transformation obtained by cutting and stacking \cite{Friedman}. This is done as follows. Let $\Omega_0 = \Omega_{0,0}$ denote the unit interval. At stage one of the construction we divide $\Omega_0$ into $m_1$ pairwise disjoint intervals, $\Omega_{0,1}, \Omega_{1, 1} \cdots, \Omega_{m_1-1, 1}$, of lengths $p_{0,1}, p_{1,1},\cdots, p_{m_1-1, 1}$, respectively, each $p_{i,j}$ being positive. Obviously $\sum_{j=0}^{m_1-1}p_{j,1}=1$. For each $j, 0\leq j \leq m_1-2$,  we stack $a_{j,1} \geq 0$ pairwise disjoint intervals of length $p_{j,1}$ on  $\Omega_{j,1}$. Each interval is mapped linearly onto the one above it, except that $a_{j,1}$-th spacer is mapped linearly onto $\Omega_{j+1, 1}$, $0 \leq j \leq m_1-2$.
We thus get a stack of certain height $h_1$, together with a map $T$ which is defined on all intervals of the stack except the interval at the top of the stack.  Note that if $p_{j,1} \neq p_{j+1,1}$ for  some $j$, $T_1$ will not be measure preserving.This completes the first stage of the construction.\\

 At the $k$-th stage we divide the stack obtained at the the $(k-1)$-th stage
in the ratios
$$p_{0,k}, p_{1,k}, \cdots, p_{m_k-1, k}, \sum_{i=0}^{m_k-1}p_{i,k} =1, $$
where each $p_{i,k}$ is positive. The spacers are added in the usual manner by which we mean that the spacers stacked above the $j$-th column are all of the same length which is the length of the top piece of the $j$-th column. The extension of $T$ to the spacers is done linearly as usual. Note that the top of the spacers above the $j$-th column is mapped linearly onto the bottom of the $(j+1)$-th
column, so that if $p_{j,k} \neq p_{{j+1},k}$, $T$ will not be measure preserving. Note that
at the $k$-th stage the measure is defined only on the algebra ${\Gamma }_k$ generated by the levels of the $k$-th stack, except the top piece. The resulting $T$, after all the stages of the construction are completed, is defined on the space $X$ consisting increasing union of stack intervals (sans $\cap_{k=1}^\infty \Omega_{m_k-1,k}$). Let $\nu$ denote the  Lebesgue measure defined on the $\sigma$-algebra  $\Gamma $ generated by $\cup_{n=1}^\infty{\Gamma}_n$. Note that $\prod_{j=1}^k p_{m_j-1,j}$ is the measure of the top piece of the column of height $h_k$ at the end of $k^{th}$ stage of construction. We require that this goes to 0 as as $k\rightarrow 0$. This ensures that $T$ is eventually defined for almost every point of $\Omega_0$.   Note that $T$ is non-singular (see remark below), ergodic with respect to $\nu$, and $\frac{d\nu\circ T}{d\nu}$ is constant on all but the top layer of every stack. If no spacers are added at every stage of the construction, the resulting transformation will be called non-singular odometer. \\

\begin{rem}\label{rem2}

 The transformation $T$ is non-singular in the sense that $m(T^{-1}(A))$ =$ 0$ whenever $m(A) =0$.
For each $j$, let $p_{l_j, j} = \max\{\{p_{i,j}: 0 \leq i\leq m_j-1\}$.

Consider the case when $\prod_{j=1}^\infty p_{l_j,j} > 0$. Then $\sum_{j=0}^\infty (1 - p_{l_j,j}) < \infty$. Now $\lambda_k \egdef \prod_{j=1}^k p_{l_j,j}$ is the length of the largest of subinterval of $[0,1)$ which appears as a level after the $k^{th}$ stage of construction. For each $k$, let $I_k$ denote this level. Then $I_{k+1} \subset I_k$,
the length of $I_{k+1} = \lambda_{k+1} = \lambda_k\times p_{l_{k+1}, k+1}$ and $W \egdef \cap_{k=1}^\infty I_k$ has positive length = $\prod_{j=1}^\infty p_{l_j, j}$. The Lebesgue measure of $ (I_k - I_{k+1})$ is $\lambda_k - \lambda_{k+1}$. Write $L_k = (\cup_{j= -a_k}^{b_k}T^j(I_k- I_{k+1}))\cap [0,1)$, where $T^{a_k}\Omega_{0,k} = I_k$, and $b_k = h_k-1 - a_k$.
Then the lebesgue measure of $L_k$ is $1 - p_{l_{k+1}, k+1}$, so by Borel Canterlli Lemma
the Lebesgue measure of $L \egdef \limsup_{k\rightarrow \infty} L_k = \cap_{k=1}^\infty\cup_{j=k}^\infty L_j$ is zero. Now if $x \in [0,1) - L$, then $x$ is in at most finitely many $L_j$s.
This means that either $x \in W$ or for some fixed $y\in W$ and for some fixed integer $n(x)$, $x = T^{n(x)}y$. Thus we see that when $\prod_{j=1}^\infty p_{l_j,j}$is non-zero,
$T$ induces a dissipative transformation on $[0,1)$ which implies that $T$ itself is dissipative in this case. There are two subcases: (i) if $l_j =0$ for all $j$ bigger than a fixed  integer $N >0$, then $T$ is non-invertible and dissipative; $W$ is the required wandering set which admits only finitely many negative iterates, but admits all positive iterates; (ii) in case $l_j \neq 0$
for infinitely many $j$, then $T $ is invertible and dissipative, $W$ admits pairwise disjoint iterates over all integers. Note that measure is defined on the
$\sigma$-algebra generated by levels of the stacks, and we really have a discrete measure space, and ergodicity holds.\\

In case $\prod_{j=1}^\infty p_{l_j,j} =0$, then $T$ is defined on an atomfree measure space and the ergodicity of $T$ follows from the usual Lebesgue density argument.\\

 \end{rem}
\paragraph{\textbf{Unitary Operators $U_T$ and $V_\phi$.}}

Let $\phi$ be a function on $X$ of absolute value 1 which is constant  on all but the top layer of every stack. On $L^2(X, \Gamma,\nu)$ define
$$(U_T f)(x)  = \Big(\frac{d\nu\circ T}{d\nu}(x)\Big)^{1/2} f(Tx), f \in L^2(X, \Gamma , \nu)$$
$$(V_\phi f)(x)  = (Vf)(x) = \phi(x)\cdot (U_Tf)(x), f \in L^2(X, \Gamma, \nu).$$
$U_T$, and $V$ are unitary operators, except  when $\prod_{k=1}^\infty p_{0,k} > 0$, in which case
$U_T, V_\phi$ are isometries isomorphic to the shift on $l^2$. The following argument,
which is for the case when $V$ is unitary, can be modified suitably to cover the case
of when is an an isometry.

$$(U^n_Tf)(x) = \Big(\frac{d\nu\circ T^n}{d\nu}(x)\Big)^{1/2}f(T^nx),$$

$$(V^nf)(x)=\prod_{j=0}^{n-1}\phi(T^j(x))\Big(\frac{d\nu\circ T^n}{d\nu}(x)\Big)^{1/2}f(T^nx).$$

It is known \cite{Nadkarni1} that the $V$ has simple spectrum whose maximal spectral type (except possibly for some discrete part) is given by the generalized Riesz product
$$\prod_{j=1}^\infty p_{0,j}\mid P_j(z)\mid^2,$$
where
$$P_j(z) = 1 + c_{1,j}\Big(\frac{p_{1,j}}{p_{0,j}}\Big)^{1/2}z^{-R_{1,j}} +\cdots + c_{m_j-1,j}\Big(\frac{p_{m_j-1,j}}{p_{0,j}}\Big)^{1/2}z^{-R_{m_j-1,j}}$$

The constants $c_{i,j}, 1\leq i \leq m_j-1, j = 1,2,\cdots$, are of absolute value 1. They are determined by $\phi$. The exponent $R_{i,j}, 1 \leq i \leq m_j-1$,$ j =1,2,\cdots$, is the
$i$-th return time of a point in $\Omega_{0,j}$ into $\Omega_{0,j-1}$. It is equal to
$$R_{i,j} = ih_{j-1} + a_{0,j} + a_{1,j} + \cdots +a_{i-1,j}, 1 \leq i \leq m_{j}-1$$
We give the steps involved in proving this as it will allow us to make some needed observations. Write $Tf = f\circ T$. We have

$$(V^{-n}f)(\cdot) = T^{-n}\circ\Big(\Big(\prod_{j=0}^{n-1}\phi(T^j(\cdot))\Big)^{-1}(\frac{d(\nu\circ T^n)}{d\nu}(\cdot))^{-1/2}f(\cdot)\Big)$$
$$= \Big(\prod_{j=0}^{n-1}\phi (T^{j-n}(\cdot))\Big)^{-1}\Big(\frac{d\nu\circ T^n}{d\nu}(T^{-n}(\cdot))\Big)^{-1/2}f(T^{-n}(\cdot)),$$

whence
$$(T^{-n}f)(\cdot) =  \prod_{j=0}^{n-1}\phi (T^{j-n}(\cdot))\Big(\frac{d(\nu\circ T^n)}{d\nu}\Big)^{1/2}(T^{-n}(\cdot))(V^{-n}f)(\cdot)$$

Let $\Omega_{0,k-1}$ denote the base of the stack of height $h_{k-1}$ after $(k-1)$-th stage of construction. Let $\Omega_{0,k}, \Omega_{1,k}, \cdots, \Omega_{m_k-1, k}$ be the partition of $\Omega_{0,k-1}$ during the $k$-th stage of construction, and let $a_{i,k}$ denote the number of spacers put on the column with base $\Omega_{j, k}$, $0 \leq j < m_k-1$. We have
$$\Omega_{0,k-1} = \cup_{j=0}^{m_k-1}T^{jh_{k-1}+ \sum_{i=0}^{j-1}a_{i,k}}(\Omega_{0,k})$$
$$= \cup_{j=0}^{m_k-1}T^{R_{j,k}}(\Omega_{0,k})$$
where

$$R_{j,k} = jh_{k-1} + \sum_{i=0}^{j-1}a_{i,k}$$

$$= \rm{the}~{j-}{\rm{th}} ~~ {\rm{return ~~time ~~of ~~a ~~point ~~in}}~~\Omega_{0,k} ~~{\rm{into}}~~\Omega_{0,k-1}.$$
$$1_{\Omega_{0,k-1}} = \sum_{j=0}^{m_k-1}1_{\Omega_{0,k}}\circ T^{-R_{j,k}},$$
$$=\sum_{j=0}^{{m_k-1}}c_{j,k}\Big(\frac{d\nu\circ T^{R_{j,k}}}{d\nu}(T^{-R_{j,k}})\Big)^{1/2}(\cdot)(V^{-R_{j,k}}1_{\Omega_{0,k}})(\cdot)$$
where $ c_{j,k} =\displaystyle \prod_{j=0}^{R_{j,k}-1}\phi (T^{j-R_{j,k}}(\cdot))$, a constant of absolute value one. Note that the constants $c_{j,k}$ can be preassigned and $\phi$ can be so defined that the above relation holds for all $(j,k)$.
We now observe that for $x\notin T^{R_{j,k}}\Omega_{0,k}$,
$$V^{-R_{j,k}}1_{\Omega_{0,k}}(x) = 0,$$  and that for $x\in T^{R_{j,k}}\Omega_{0,k}$,
$$\frac{d\nu\circ T^{R_{j,k}}}{d\nu}(T^{-R_{j,k}}(x)) = \frac{p_{j,k}}{p_{0,k}}.$$
We thus have
$$1_{\Omega_{0,k-1}}  =\sum_{j=0}^{{m_k-1}}c_{j,k}\Bigg(\frac{p_{j,k}}{p_{0,k}}\Bigg)^{1/2}(V^{-R_{j,k}}1_{\Omega_{0,k}})(\cdot)$$
Let us normalize $1_{\Omega_{0,k}}$ and write
$$f_k = \Big(\frac{1}{ m(\Omega_{0,k})}\Big)^{1/2}1_{\Omega_{0,k}} = \Bigg(\frac{1}{(\prod_{j=1}^kp_{0,j})}\Bigg)^{1/2}1_{\Omega_{0,k}}$$
$$f_{k-1} = (p_{0,k})^{1/2}\Bigg(1 + c_{1,k}\Big(\frac{p_{1,k}}{p_{0,k}}\Big)^{1/2}V^{-R_{1,k}} + \cdots + c_{m_k-1,k}\Big(\frac{p_{m_k-1,k}}{p_{0,k}}\Big)^{1/2}V^{-R_{m_k-1,k}}\Bigg)f_k$$
Now $m(\Omega_{0,0}) = 1$ so $f_0 = 1_{\Omega_{0,0}}$. We have by iteration
$$f_0 = (\prod_{j=1}^kP_j(V))f_k,$$
where
$$P_j(z) = (p_{0,j})^{1/2}\Bigg(1 + c_{1,j}\Big(\frac{p_{1,j}}{p_{0,j}}\Big)^{1/2}z^{-R_{1,j}} + \cdots + c_{m_j-1,j}\Big(\frac{p_{m_j-1,j}}{p_{0,j}}\Big)^{1/2}z^{-R_{m_j-1,j}}\Bigg)$$

Let $V^n = \int_{S^1}z^{-n}dE, n \in \mathbb Z$, be the spectral resolution of the unitary group $V^n, n \in \mathbb Z$, and let

$$ (V^nf_k, f_k) = \int_{S^1}z^{-n}(E(dz)f_k,f_k) =\int_{S^1}z^{-n}d\sigma_k$$

where $\sigma_k (\cdot) = (E(\cdot)f_k, f_k)$\\

We therefore have for all integers $l$

$$(V^lf_0,f_0) = \int_{S^1}z^{-l}d\sigma_0 =\int_{S^1}z^{-l} \prod_{j=0}^k\mid P_j(z)\mid^2d\sigma_k,$$

whence
$$d\sigma_{0} = \prod_{j=1}^k\mid P_j(z)\mid^2d\sigma_k$$

Now we will show, as in the measure preserving case \cite{Nadkarni1}, that $\sigma_0$ is the generalized Riesz product:

$$\sigma_0 = \prod_{j=1}^\infty\mid P_j(z)\mid^2 .$$

Let $N_k$ denote the the set of integers consisting of zero together with the entry times
of a point in $\Omega_{0,k}$ into $\Omega_{0,0}$ which are less than the height $h_k$ of the $k^{th}$ stack.

We have
 $$f_0 = \Big(\prod_{j=1}^kP_j(V)\Big)f_k = Q_k(V)f_k,$$
where
$$Q_k(z) = \prod_{j=1}^k P_j(z) \egdef \sum_{j=0}^{h_k-1}q_j(k)z^j = \sum_{j \in N_k}q_j(k)z^j.$$
an expansion of the product of dissociated polynomials $P_1, P_2, \cdots, P_k$
Note that for
\begin{enumerate}[(i)]
\item $\mid q_j\mid \leq 1$ and
 \item $\mid q_r\mid \leq  \prod_{j=1}^{k-1} p_{m_{j-1},j} \rightarrow 0$ as $k\rightarrow \infty$, for $ h_k - h_{k-1} < r < h_k$,
\end{enumerate}
$$\int_{S^1}z^{n}\mid Q_k\mid^2dz = \sum_{r-s+n=0}q_r(k)\overline {q_s(k)}$$
where $r,s \in N_k$.

Now fix $n  \in \mathbb Z$ and let $k$ be so large that the first return time for any $x\in \Omega_{0,k}$ back to $\Omega_{0,k}$ is bigger than $\mid n\mid$, i.e, $k$ is so large that
$h_k \geq \mid n\mid$. We actually choose $k$ so large that $\mid n\mid < \frac{h_{k-1}}{2}$. If $r,s \in N_k$ then $r-s+n $ can never exceed or equal the second return of of an $x \in \Omega_{0,k}$ back to $\Omega_{0,k}$ (under $T$ or $T^{-1})$. Moreover there can be at most $n^2$ pairs $(r,s)$ with $r,s \in N_k$ with $T^{r+n-s}\Omega_{0,k}\cap \Omega_{0,k} \neq \emptyset$. For suppose $n >0$ and $T^{r+n-s}\Omega_{0,k}\cap \Omega_{0,k} \neq \emptyset$ and $r-n-r \neq 0$, $r,s \in N_k$. Then $r+n-s=u$ where $u$ is the first return time of a point $x \in \Omega_{0,k}$ back to $\Omega_{0,k} \geq h_k$.
$s =r+n-u$. Since $n,r,s$ are less than $h_k$, $h_k \leq u$ and $s \geq 0$, we have $0 \leq  s <n$ and $n-s +r = u > h_k,$ so $r \geq h_k -(n-s) \geq  h_k -n$. Thus there can be at most $n^2$ pairs $(r,s)$ with $r,s \in N_k$ such that $T^{n+r-s}\Omega_{0,k}\cap\Omega_{0,k} \neq \emptyset$. Thus if $ T^{n+r-s}\Omega_{0,k}\cap \Omega_{0,k} \neq \emptyset, r,s \in N_k$ then $n+r-s =0$ except for at most $n^2$ pairs $(r,s)$, $r,s \in N_k$. This in turn implies that $(V^{n+r-s}1_{\Omega_{0,k}},1_{\Omega_{0,k}}) = 0$ except when $n+r-s =0$  and at
most $n^2$ other pair (r,s), $r,s\in N_k$.

$$(V^nf_0,f_0) = (V^nQ_k(V)f_k,Q_k(V)f_k) = (V^n\mid Q_k\mid^2(V)f_k,f_k)$$
$$=\sum_{n+r-s=0, r,s \in N_k}q_r{\overline {q_s}}(V^{n+r-s}f_k,f_k) + \sum_1$$
$$\sum_{r-s+n=0}q_r(k)\overline {q_s(k)} + \sum_1$$

where $\sum_1$ is a sum of at most $n^2$ terms of the type
$$q_r{\overline{q_s}}(V^{n+r-s}f_k,f_k) , n+r-s \neq 0$$

Now $h_k - h_{k-1} < h_k-n \leq r \leq h_k-1$, so that $\mid q_r(k)\mid \leq \prod_{j=1}^{k-1}p_{m_{j}-1, j} \rightarrow 0$ as $k\rightarrow \infty$. Clearly then the sum $\sum_1$ goes to zero as $k\rightarrow \infty$ and the claim is proved.\\

\begin{rem}\label{rem2} Let $p_{l_j,j}$ be as in remark 6.1. Note  that if $\prod_{j=1}^\infty p_{l_j,j} > 0$ then $T$ is dissipative, so $V_\phi$ has Lebesgue spectrum. For the subcase  when, in addition,
$l_j= 0$ for all but finitely many $j$, then Mahler measure of $\sigma_0$ is positive.\\

If $\prod_{j=1}^\infty p_{l_j,j} =0$ and $V_\phi = U_T$, then it can be verified that $\sum_{k=1}^\infty\mid \hat{\sigma_0}(k)\mid^2 = \infty$. In addition if $\sum_{j=-\infty}^\infty p_{m_j-1,j}p_{0,j} = \infty$ then one  can adapt the method of I. Klemes and K. Reinhold \cite {KlemesR} to show that $\sigma_0$ is singular to Lebesgue measure.\\
\end{rem}

\paragraph{\textbf{Generalized Riesz Products of Dynamical Origin.}}
Consider now the polynomials appearing in the above generalized Riesz product.
$$P_j(z) = (p_{0,j})^{1/2}\Bigg(1 + c_{1,j}\Big(\frac{p_{1,j}}{p_{0,j}}\Big)^{1/2}z^{-R_{1,j}} + \cdots + c_{m_j-1,j}\Big(\frac{p_{m_j-1,j}}{p_{0,j}}\Big)^{1/2}z^{-R_{m_j-1,j}}\Bigg)$$

 The exponent $R_{i,j}, 1 \leq i \leq m_j-1, j =1,2,\cdots$, is the
$i$-th return time of a point in $\Omega_{0,j}$ into $\Omega_{0,j-1}$. Also
$$R_{i,j} = ih_{j-1} + a_{0,j} + a_{1,j} + \cdots +a_{i-1,j}, 1 \leq i \leq m_{j}-1$$
where $h_{j-1}$ is the height of the tower after $(j-1)$-th stage of the construction is complete, and $a_{k,j}$ is the number of spacers on the $k$-th column, $ 0 \leq k\leq m_j-2$.
We observe that\\
\begin{enumerate}
\item  $h_1 = R_{m_{1}-1,1} +1$,\\
\item  $R_{1,j} \geq h_{j-1} > R_{m_{j-1},j-1}$,\\
\item   $R_{i+1, j} - R_{i,j} \geq h_{j-1}$. \\
\end{enumerate}
  These properties (1), (2), (3) of the powers $R_{i,j}$, $1 \leq i \leq m_j -1, j=1,2,\cdots$ indeed characterize generalized Riesz products which arise from nonsingular rank one transformations (together with a $\phi$) in the above fashion. More precisely consider a generalized Riesz product
$$\prod_{j=1}^\infty\mid Q_j(z)\mid^2. $$
where $$Q_j(z) = \sum_{i=0}^{n_j} b_{i,j}z^{r_{i,j}},~~ b_{i,j} \neq  0, ~~\sum_{i=0}^{n_j}\mid b_{i,j}\mid^2 =1, \prod_{j=1}^\infty \mid b_{n_j,j}\mid =0.$$
Define inductively: $$h_0 = 1, h_1 = r_{n_1,1} +h_0, \cdots , h_j = r_{n_j,j} +h_{j-1}, j \geq 2$$
Note that $ h_j > r_{n_j,j}$.\\

\begin{Prop}\label{prop1}
 Assume that for each $j=1,2,\cdots$,
$$r_{1,j} \geq h_{j-1}, ~~~r_{i+1,j} - r_{i,j} \geq h_{j-1}$$
Then $r_{i,j}, h_j$, satisfy (1), (2) and (3). The generalized product $\prod_{j=1}^\infty \mid Q_j\mid^2$ describes the maximal spectral type (up to possibly a discrete part) of a suitable $V_\phi$.
\end{Prop}
\begin{proof}
 That the $r_{i,j}, h_j$ satisfy (1), (2), (3) is obvious. The needed non-singular $T$ is given by cutting parameters $ p_{i,j} =\mid b_{i,j}\mid^2, i= 0, 1, \cdots, n_j, j = 1,2,\cdots$, and spacers $a_{i-1,j} = r_{i,j} - r_{i-1,j}- h_{j-1}$, $1 \leq i \leq n_j, j =1,2,\cdots$. The needed $\phi$ (which need not be unique) is given by constants $\frac{b_{i,j}}{\mid b_{i,j}\mid}, 0\leq i \leq n_j, j =1,2,\cdots$. This proves the proposition.\\
\end{proof}

\begin{Def}\label{def1}
 A generalized Riesz product $\mu = \prod_{j=1}^\infty\mid Q_j(z)\mid^2$,
where $Q_j(z) = \sum_{i=0}^{n_j} b_{i,j}z^{r_{i,j}}, b_{i,j} \neq 0, \sum_{i=0}^{n_j}\mid b_{i,j}\mid^2 =1$, $\prod_{j=1}^\infty\mid b_{n_j, j}\mid =0$, is said to be of dynamical origin if
 with $$h_0 = 1, h_1 = r_{n_1,1} +h_0, \cdots , h_j = r_{n_j,j} +h_{j-1}, j \geq 2$$
it is true that for  $j=1,2,\cdots$,
$$r_{1,j} \geq h_{j-1}, ~~~r_{i+1,j} - r_{i,j} \geq h_{j-1}$$
If, in addition, the coefficients $b_{i,j}$ are all positive, then we say that $\mu$ is of purely dynamical origin.\\
\end{Def}

 \paragraph{\textbf{Flat Polynomials and Generalized Riesz Products.}}

 \begin{lem}\label{lem1}

Given a sequence $P_n = \sum_{j=0}^{m_n} a_{j, n}z^j, , n=1,2,\cdots$ of analytic trigonometric polynomials in $L^2(S^1,dz)$ with non-zero constant terms and $L^2(S^1,dz)$ norm 1, $\prod_{n=1}^\infty \mid a_{m_n, n}\mid =0$, there exist a sequence of positive integers $N_1, N_2, \cdots$ such that
$$\prod_{j=1}^\infty\mid P_j(z^{N_j})\mid^2$$
is a generalized Riesz product of dynamical origin.
\end{lem}
\begin{proof}
 For each $j\geq 1$, let
$$P_j = \sum_{i=0}^{n_j}b_{i,j}z^{r_{i,j}}, b_{i,j} \neq 0,  ~~b_{0, j} > 0, ~~\sum_{i=1}^{n_j} \mid b_{i,j}\mid^2 =1.$$
Let $N_1 =1$ and $h_1 = H_1 = r_{n_1,1}+1$. Choose $N_2 \geq 2H_1  > 2r_{n_1,1}$. Then
$${N_2\cdot r_{1,2}} > h_1, N_2(r_{i+1,2} - r_{i,2}) > h_1.$$

 Since $N_2 > 2r_{n_1,1}$ the   polynomials $\mid P_1(z^{N_1})\mid^2$ and $\mid P_2(z^{N_2})\mid^2$ are dissociated.
Consider now  $P_1(z^{N_1})P_2(z^{N_2})$. Write $H_2 = N_1r_{n_1,1} + N_2r_{n_2,2} +  h_1> N_2r_{n_2, 2} + h_1 \egdef   h_2 $.
Choose $N_3 \geq 2H_2$.
Then
$$N_3\cdot r_{1, 3} \geq h_2, N_3(r_{i+1, 3} -r_{i,3}) > h_2.$$
Since $N_3 \geq  2 H_2 > 2(N_1r(n_1,1) + N_2r(n_2,2)$ the polynomial $\mid P_3(z^{N_3})\mid^2$ is dissociated from $\mid P_1(z^{N_1})P_2(z^{N_2})\mid^2$.
Proceeding thus we get $N_j, j =1,2,\cdots$ and  polynomials  $Q_j(z) = P_j(z^{N_j}), j  =1,2,\cdots$ such that

(i) $\mid\mid Q_j \mid\mid_2 = 1$ (since $\mid\mid P_j\mid\mid_2 = 1$ and the map
$z \rightarrow z^{N_j}$ is measure preserving.)
(ii) the polynomials $\mid Q_j\mid^2, j =1,2,\cdots$ are dissociated,
(iii) for each $j\geq 1$,
$$h_{j-1} < N_{j}r_{1, j}, ~~h_{j-1} < N_{j}(r_{i+1,j} - r_{i, j}) $$

Since the polynomials $Q_j, j =1,2,\cdots$ have $L^2(S^1,dz)$ norm 1 and their absolute squares are dissociated, the generalized Riesz product $\prod_{j=1}^\infty\mid P(z^N_j)\mid^2$
is well defined. Moreover, (iii) shows that the conditions for it to arise  from a
non-singular rank one $T$ and a $\phi$ in the above fashion are satisfied. The lemma follows.
\end{proof}

An immediate application of this Lemma is the following:\\
\begin{Th}\label{th7}
 {\it Let $P_j, j =1,2,\cdots$ be a sequence of analytic trigonometric polynomials satisfying the conditions of lemma 6.5 and  such that $\mid P_j(z)\mid \rightarrow 1 $ a.e. $(dz)$ as $j \rightarrow \infty$. Then there exists a subsequence $P_{j_k}, k=1,2,\cdots$ and natural numbers $N_1 < N_2 <  \cdots$ such that the product
$\mu =\prod_{k=1}^\infty \mid P_{j_k}(z^{N_k})\mid^2$  is a generalized Riesz product of dynamical origin with $\frac{d\mu}{dz} > 0$ a.e. $(dz)$. }
\end{Th}
\begin{proof}
 Since $\mid P_j(z)\mid \rightarrow 1$ as $j \rightarrow \infty$ a.e. $(dz)$, by Egorov's theorem we can extract a subsequence $P_{j_k}, k = 1,2,\cdots$ such that
the sets $$E_k \egdef \Big\{z: \mid (1- \mid P_{j_l}(z)\mid )\mid < \frac{1}{2^k} ~~\forall ~~l \geq k \Bigg\} $$
increase to $S^1$ (except for a $dz$ null set), and $\sum_{k=1}^\infty (1 -dz(E_k)) < \infty.$
 Write $Q_k = P_{j_k}$. Then for $z \in E_k$, $\mid (1 - \mid Q_k(z)\mid)\mid < \frac{1}{2^k}$.  By the lemma above we can choose $N_1, N_2, \cdots$ such that
$$\prod_{k=1}^\infty \mid Q_k(z^{N_k}) \mid^2$$ is a generalized Riesz product of dynamical origin. We show that $\displaystyle \lim_{L\rightarrow \infty}\prod_{k=1}^L|Q_k(z^{N_k})|$
is nonzero a.e. $(dz)$, which will imply, by proposition 5.2., that $\frac{d\mu}{dz} > 0$ a.e $(dz)$.\\

 Now the maps $S_k: z \rightarrow z^k, k=1,2,\cdots$ preserve the measure $(dz)$, and since $\sum_{k=1}^\infty dz(S^1 -E_k) < \infty$ we have
 $\sum_{k=1}^\infty dz(S^{-N_k}(S^1 - E_k)) < \infty$. Let $F_k = S^{-N_k}(S^1 - E_k)$ and
 $F = \limsup_{k\rightarrow \infty} F_k = \cap_{k=1}^\infty\cup_{l=k}^\infty F_l$.
 Then $dz(F) =0$, and if $z\notin F$, $z \notin S^{-N_k}(S^1 - E_k)$ hold for all but
 finitely many $k$, which in turn implies that $S^{N_k}z \in E_k$ for all but finitely many $k$.
 Thus, if $z \notin F$, then $\mid (1 - \mid Q_k(z^{N_k})\mid)\mid < \frac{1}{2^k}$ for all but finitely many $k$. Also the set of points $z$ for which some finite product $\prod_{k=1}^L\mid Q_k(z^{N_k})\mid$ vanishes is countable. Clearly
$\lim_{L\rightarrow \infty}\prod_{k=1}^L\mid Q_k(z^{N_k})\mid$
is nonzero a.e. $(dz)$ and the theorem is proved.\\
\end{proof}

\begin{Cor}\label{cor1}
 (i) If $P_k, k=1,2,\cdots$  are as in the above theorem and if\linebreak $\limsup_{k\rightarrow \infty} M(P_k) =1$, then we can choose $P_{j_k}, k=1,2,\cdots$ and $N_1, N_2, \cdots, $ in such a way that $M(\mu) $ is positive.\\
 (ii) If $P_k, k =1,2,\cdots$ are as in the above theorem  and  if $\liminf_{k\rightarrow \infty}M(P_k) < 1$  then we can choose $P_{j_k}, k =1,2,\cdots$ and $N_1, N_2, \cdots $ in such a way that $M(\mu) =0$, and $\frac{d\mu}{dz} > 0$ a.e. $(dz)$.

\end{Cor}

\begin{rem}\label{rem2}

  Now it is easy to construct polynomials $P_k, k=1,2,\cdots$ satisfying the hypothesis of
part (ii) of the above corollary, so one can obtain generalized Riesz products $\mu$ with
zero Mahler measure and $\frac{d\mu}{dz}$ positive a.e $(dz)$. Thus the interchange of order of limits suggested in remark 5.5. is therefore not valid without some additional conditions.\\

Let (U) denote the class of all unimodular polynomials, i.e., polynomials of the type  $$\Biggl\{\sum_{j=0}^na_jz^j: \mid a_j\mid =1, 0 \leq j\leq n, n \geq 1\Biggr\}.$$ Note that $\| \frac{1}{\sqrt {n+1}} P\|_2 =1$ for any polynomial in the class (U) of degree $n$. A question of Littlewood, answered in the affirmative by J-P. Kahane \cite{Kahane}, \cite{QS}, asks if there is a sequence $P_j, j=1,2,\cdots $ of polynomials in the class $(U)$
 such that $\frac{1}{\sqrt {d_{j}+1}}| P_j|,$ $j =1,2,\cdots$  converges to the constant function 1 uniformly on $S^1$, where $d_j$ is the degree of the polynomial $P_j$, $j =1,2,\cdots$.
 Littlewood problem remains open if we require that the coefficients of $P_j$ be either -1 or 1, for all $j$. Let $P_j, j =1,2,\cdots $ be a sequence Kahane polynomials. Clearly then $\int_{S^1}\log |P_j(z)| dz \rightarrow 0$ as $j \rightarrow \infty$, which in turn implies that the Mahler measure of $P_j$ converges to 1 as $j\rightarrow \infty$. Thus the sequence of Kahane polynomials satisfies the conditions of part (i) of above theorem and we see that
 Kahane polynomials give rise to  generalized Riesz products of Dynamical origin absolutely continuous with respect to the Lebesgue measure and with positive Mahler measure.\\

\end{rem}
 A sequence of polynomials $P_k, k=1,2,\cdots$  in the class $U$ is said to be ultraflat if  $\frac{\mid P_k\mid }{\mid\mid P_k\mid\mid_2}, k =1,2,\cdots$ converge uniformly to the constant function 1.

 As mentioned above it is not known if there is a sequence of ultra flat polynomial with coefficients $+1$ and $-1$. However, M. Guenais\cite{Guenais} has shown that there is a sequence $P_k, k=1,2,\cdots$ of polynomials in $U$ with coefficients in $\{-1,1\}$ such that $\frac{\mid P_k\mid}{\mid\mid P_k\mid\mid_2}\rightarrow 1$ a.e.($dz$) if and only if there is a measure preserving general odometer action $T$  and $\phi$ taking values in $\{-1,1\}$ such that $V_\phi$ has Lebesgue spectrum. Here $\phi$ has to be of the special kind described above, namely, it is constant on all but the top level of the stacks associated to $T$.\\

A finite sequence $(e_0, e_2, \cdots, e_{n-1})$ of $+1$ and $-1$ is said to be a Barker
sequence if for all $k, 0 < k \leq  n-1 $  the aperiodic correlation
$$ \sum_{j=0}^{n-k}e_je_{j+k}$$
does not exceed 1 in absolute value.\\

It is known that there are only finitely many Barker sequences of odd length, and there are
no Barker sequences of odd length greater than 13. For more information on Barker sequences
and their significance in Radar signal processing theory we refer the reader to \cite{down}, \cite{bor}, \cite{bor1}. It is not known if there are infinitely many Barker sequences, and it is conjectured that there are only finitely many Barker sequences. P. Borwein and M. Mossinghoff\cite{bor}  have shown that if $e_0, e_1, \cdots, e_{n-1}$ is a Barker sequence of length $n$ and if
$$P(z) = \frac{\sum_{j=0}^{n-1}e_jz^j}{\sqrt n},$$ then
$M(P) > 1-\frac{1}{n}$. This immediately implies, in the light of the result of M. Guenais, or by the corollary above the following theorem.
\begin{Th}\label{th7}
 If there are infinitely many Barker sequences then there is a generalized
Riesz product $ \prod_{j=1}^\infty \mid P_j\mid^2$  of dynamical origin with measure preserving $T$, with positive Mahler measure, and such that coefficients of each $P_j$ are real and equal in absolute value. The measure preserving $T$ can be chosen to be an odometer action.
\end{Th}

One can ask the question if there is a sequence $P_k, k=1,2,\cdots$ of polynomials with coefficients in $\{-1, 0, 1\}$ such that
 $\frac{\mid P_k\mid}{\mid \mid P_k\mid\mid_2}\rightarrow 1~~\rm{a.e.}~(dz) ~~{\rm{as}} ~~k\rightarrow \infty.$ Theorem 6.5. at once implies that this is possible if and only if  there is a generalized Riesz  product $\mu = \prod_{j=1}^\infty\mid Q_j\mid^2$, $\frac{d\mu}{dz} > 0$ a.e. $(dz)$,  of dynamical origin and such that for each $j$ the non-zero coefficients of $Q_j$ are real and equal in absolute value.\\

Consider the class of (B) of all polynomials of the type
 $$P(z) = \frac{1}{\sqrt{m+1}}(1 + z^{n_1} + z^{n_2} + \cdots +z^{n_m}),$$ where $0 <n_1 < n_2 < \cdots < n_m$.
 Since $L^2(S^1,dz)$ norm of such a $P $  is one, its $L^1(S^1,dz)$-norm is at most one. J. Bourgain\cite{bourgain} has raised the question if the supremum of the $L^1(S^1,dz)$-norms of elements in  $(B)$ can be 1, see \cite{bourgain}. We have the following result due to M. Guenais, proved here more generally than in \cite{Guenais}.

\begin{Prop}\label{prop1}
 Let $\mu = \prod_{k=1}^\infty \mid P(z)\mid^2$ be a generalized
Riesz product. If  $\sum_{k=1}^\infty ( 1 - \mid\mid P_n\mid\mid_1^2)^{1/2}$ is finite then $\frac{d\mu}{dz} > 0$ on a set of positive Lebesgue measure in $S^1$.
\end{Prop}
\begin{proof}
Write $v_k^2 =  1 - \mid\mid P_k\mid\mid_1^2$. Then
$\sum_{k=1}^\infty v_k < \infty$, equivalently\linebreak $\prod_{k=1}^\infty\mid\mid P_k\mid\mid_1 >0$. For all functions $f,g \in L^2(S^1, dz)$, Cauchy-Schwarz inequality gives

$$\mid(\mid\mid f\cdot g\mid\mid_1   - \mid\mid f\mid\mid_1 \mid\mid g\mid\mid_1)\mid \leq (\mid\mid f\mid\mid_2^2 -\mid\mid f\mid\mid_1^2)^{1/2} (\mid\mid g\mid\mid_2^2 -\mid\mid g\mid\mid_1^2)^{1/2}.$$

Fix an integer $n_0 > 1$ and let $k > n_0$. Then

$$\mid( \mid \mid\prod_{j=n_0}^k P_j\mid\mid_1 - \mid\mid\prod_{j=n_0}^{k-1}P_j\mid\mid_1\mid\mid P_k\mid\mid_1)\mid$$
$$\leq (\mid\mid \prod_{j=n_0}^{k-1}P_j\mid\mid_2^2 - \mid\mid \prod_{j=n_0}^{k-1} P_j\mid\mid_1^2)^{1/2} ( \mid\mid P_k\mid\mid_2^2 - \mid\mid P_k\mid\mid_1^2)^{1/2}$$
$$\leq v_k.$$
 So,
 $$ \mid(\mid \mid\prod_{j=n_0}^k P_j\mid\mid_1 - \mid\mid\prod_{j=n_0}^{k-1}P_j\mid\mid_1\mid\mid P_k\mid\mid_1)\mid \leq v_k.$$

$$\mid(\mid\mid \prod_{j=n_0}^{k-1} P_j\mid\mid_1\mid\mid P_k\mid\mid_1 - (\mid\mid\prod_{j=n_0}^{k-2}P_j\mid\mid_1)(\mid\mid P_{k-1}\mid\mid_1)\mid\mid P_k\mid\mid_1)\mid \leq v_{k-1}\mid\mid P_k\mid\mid_1 \leq v_{k-1}$$
                       $${  \vdots   ~~~~ \vdots    ~~~~  \vdots}$$
$$\mid(\mid\mid \prod_{j=n_0}^{n_0 +1}P_j\mid\mid_1\prod_{j=n_0+2}^k\mid\mid P_j\mid\mid_1- \prod_{j=n_0}^k\mid\mid P_j\mid\mid_1)\mid \leq v_{n_0+1}.$$

On adding the above inequalities:
$$\mid( \mid\mid\prod_{j=n_0}^kP_i\mid\mid_1 - \prod_{j=n_0}^k\mid\mid P_j\mid\mid_1)\mid \leq \sum_{j=n_0}^kv_j$$

Since $\prod_{j=1}^\infty \|P_j\|
_1 > 0$ and $\sum_{j=1}^\infty v_k < \infty$,
we see that\linebreak $\displaystyle \limsup_{k\rightarrow \infty}\mid\mid \prod_{j=1}^k P_j\mid\mid_1  > 0$,
so by Bourgain's criterion for singularity (i.e., corollary 4.4 with $Q_j =1$ for all $j$,)  we see that $\mu$ is not singular to Lebesgue measure
on $S^1$.
\end{proof}

The only known rank one non-singular $T$ for which $U_T$ has Lebesgue spectrum is the
 one where the cutting parameter satisfies $\prod_{j=1}^\infty p_{l_j,j} > 0$. One can ask  if there exists a non-dissipative  rank one transformation whose maximal spectral admits a component equivalent to the Lebesgue measure on $S^1$. We discuss this question in the light of the above considerations. Fix $0 < \lambda < 1$ and let

 $$A_{\lambda}= \Bigg\{\sum_{j=0}^na_jz^j: \forall j, 0 \leq a_j \leq \lambda,
 \sum_{j=1}^n\mid a_j\mid^2 =1, n =1,2,\cdots\Bigg\},$$

 $$\alpha_\lambda = \sup_{P\in A_\lambda}\mid\mid P\mid\mid_1$$
\begin{Prop}\label{prop1}
  If $\alpha_\lambda =1$ for some $\lambda, 0 < \lambda < 1$, then there is a non-dissipative non-singular rank one map $T$ such $U_T$ has absolutely continuous part (w.r.t $(dz)$) in its spectrum.
\end{Prop}
\begin{proof}

 This follows from lemma 6.4 and proposition 6.10.\\
\end{proof}

\begin{Prop}\label{prop1}
 If there is a sequence $P_n, n=1,2,\cdots$, of polynomials in $A_\lambda$  such that $\displaystyle\lim_{n\rightarrow \infty}\mid P_n(z)\mid = 1$ a.e.$(dz)$, then there is a non-dissipative non-singular rank one $T$ such that the spectrum of $U_T$ has a part equivalent to the Lebesgue measure on $S^1$.
\end{Prop}
\begin{proof}
 This follows from theorem 6.6.\\
\end{proof}

There is a partial converse to proposition 6.11. Let  $\mu = \prod_{j=1}^\infty \mid P_j\mid^2$ be a a generalized Riesz product of class (L), with each $P_j \in A_\lambda$, and $P_j's$ dissociated.
\begin{Prop}\label{ prop1}
If $\frac{d\mu}{dz} > 0$ a.e. $(dz)$, then $\alpha_\lambda =1$.
\end{Prop}
\begin{proof}
We know from proposition 5.2. $$\lim_{k\rightarrow \infty}\prod_{j=1}^k\mid P_j(z)\mid =\frac{d\mu}{dz} ~~a.e. ~~(dz).$$
Hence there is a sequence $n_1 < n_2 <\cdots$ such that $$\prod_{j=n_k+1}^{n_{k+1}}\mid P_j(z)\mid \rightarrow 1 ~~{\rm{a.e.}} ~~(dz) ~~{\rm{as}} ~~k\rightarrow \infty$$
Since $P_j$'s are dissociated,each finite product is in $A_\lambda$, the proposition follows from Fatou's lemma.
 \end{proof}

The following three problems about the class $A_\lambda$ are thus intimately related to spectral questions about invertible non-singular rank one transformations: 
\begin{enumerate}[(1)]
\item is $\displaystyle \sup_{P\in A_\lambda}\mid\mid P\mid\mid_1 =1$ ?,
\item is  $\displaystyle \sup_{P\in A_\lambda} M(P) =1$ ?,
\item is there a sequence $P_n, n=1,2,\cdots $ in $A_\lambda$ such that $\mid P_n(z)\mid \rightarrow 1$ a.e. $(dz)$ as $n\rightarrow \infty$
\end{enumerate}    

Concerning Bourgain's question, it is known that $\sup_{P\in R}\int_{S^1}\mid P(z)\mid dz \geq \frac{\sqrt \pi}{2}$. Indeed, let  $P_n, n =1,2,\cdots$ be the polynomials as in the generalized Riesz product associated to the rank one map. Put
$X_{n,j}(z)=z^{j h_n+a_{1,n}+\cdots + a_{n,j}}$. $X_{n,j}$ is a random variables. Since
$||P_n||_2=1,$ the random variables $P_n$ are uniformly integrable. In \cite{elabdaletds}, \cite{AH1} and \cite{AH2}, the authors, proved that there is a subclass of $P_n, n =1,2,\cdots,$ for which $P_n, n =1,2,\cdots$ converges in distribution to the complex Gaussian measure
${\mathcal{N}}_{\mathbb {C}}(0,1)$ on ${\mathbb {C}}$, that is,
$$dz\{ P_n \in A\} \tend{n}{+\infty}\int_A \frac1{\pi}e^{-|z|^2} dz.$$
Denote by $\mathcal{D}(P_n)$ the distribution of $P_n$.
It follows from the Standard Moment Theorem \cite[pp.100]{Chung} that
\[
||P_n||_1 =\int |P_n| dz =\int |w| d\mathcal{D}(P_n)(w) \tend{n}{\infty} \int |w| \frac1{\pi}e^{-|w|^2} dw=\frac{\sqrt{\pi}}2,
\]
$dw$ is the usual Lebesgue measure on ${\mathbb {C}}$, that is, $dx\cdot dy=rdrd\theta$.\\

Let $E = \{z: \frac{d\mu}{dz}(z) > 0\}$, where $\mu$ is a generalized Riesz product.
We will give an upper estimate of $dz (E)$.

\begin{Th}\label{th9}
Let $\mu = \prod_{j=1}^\infty \mid P_i\mid^2$ be of class (L). Let $E = \{z: \frac{d\mu}{dz}(z) > 0\}$. If $dz (E) =1$ then there is a flat sequence of finite subproducts of $P_j$'s. If $dz (E)$ is less than 1,
   then $dz (E) \leq d$, where $d$ is the liminf of $L^1(S^1,dz)$  norms of all finite subproducts of $P_j$'s.
\end{Th}
\begin{proof}
The first part follows from the discussion above. We consider the second part.
Let $ a \egdef \sup \| P_{i_1}P_{i_2}\cdots P_{i_l}\|_1$, where the supremum is taken over all finite sequences
of increasing natural numbers $i_1< i_2<\cdots <i_l$. By Fatou's
lemma we know that $\left\|\sqrt {\frac{d\mu}{dz}}\right\|_1 \leq a$.
Take an infinite  subset ${\mathcal K}_1$ of natural numbers such that its complement ${\mathcal K}_2$ within
natural numbers is also infinite. Let $\mu_1$ and $\mu_2$
be the Riesz subproducts of $\mu$ over ${\mathcal K}_1$ and ${\mathcal K}_2$ respectively. Then  by corollary \ref{cor4}

$$\Bigl(\frac{d\mu}{dz}\Bigr)^{\frac{1}{4}}  = \Bigl(\frac{d\mu_1}{dz}\Bigr)^{\frac{1}{4}} \Bigl(\frac{d\mu_2}{dz}\Bigr)^{\frac{1}{4}}$$

By Cauchy-Schwarz inequality we get

    $$ \bigintss_{S^1}\Bigl(\frac{d\mu}{dz}\Bigr)^{\frac{1}{4}}dz \leq \sqrt a\sqrt a = a$$

   In general, by iterating,

   $$\bigintss_{S^1}\Bigl(\frac{d\mu}{dz}\Bigr)^{\frac{1}{2^n}}dz \leq a \eqno (2)$$

   Letting $n\rightarrow \infty $, we see that $dz (E) \leq a$

   Let  $$d \stackrel{\textrm{def}}{=} \liminf\Biggl\{\bigintss_{S^1}\mid P_{i_1}P_{i_2}\cdots P_{i_k}\mid: {i_1<i_2 <\cdots i_k}, k=1,2,\cdots\Biggr\}$$

   Now for any $\eta >0$, considerations leading to  equation (2) above can be applied to a suitable subproduct, say $\mu_1$, over a set ${\mathcal K}_1$ of natural numbers,  so that
   $$dz\Bigl\{z: \frac{d\mu_1}{dz}(z)> 0\Bigr\} \leq d + \eta.$$ By formula (1)
   we see that $dz (E) \leq d +\eta$. Since $\eta$ is arbitrary, we have $ dz(E) \leq d$.
\end{proof}

 In connection with the discussion above, we have the following:
 \begin{Th}\label{th10}
 Let $\mu = \prod_{j=1}^\infty \mid P_j\mid^2$ be of class (L) and assume
 that
 \begin{enumerate}
   \item $\|P_j\|_1 \tend{j}{+\infty} c$, $c \in[0,1[$ and,
   \item for any continuous function $g$ on $\T$, we have
   $$\bigintss g |P_j| dz \tend{j}{+\infty} c \bigintss g dz.$$
 Then $\mu$ is singular.
 \end{enumerate}
 \end{Th}
 \begin{proof} The sequence $||P_j|-1|$ is bounded in $L^2(dz)$. It follows
 that there exists a function $\phi$ in $L^2(dz)$ such that $||P_j|-1|$
 converge weakly over a subsequence, say $n_j, j=1, 2, \cdots$, to $\phi$
 (without loss of generality we assume that $n_j = j, j = 1,2, \cdots$). It
 is shown in \cite{elabdaletds} that the measure $\phi(z) dz$ is singular
 with respect to $\mu$. According to our assumptions, we further have that
 $\phi(z) dz$ is equivalent to Lebesgue measure. Indeed, for any nonnegative
 continuous function $g$ on $\T$, we have
 $$
 \bigintss g ||P_j|-1| dz \geq \bigintss g dz-\bigintss g|P_j| dz.
 $$
 Hence, by taking the limit combined with our assumptions, we get
 $$\bigintss g \phi dz \geq (1-c) \bigintss g dz,$$
 which finish the proof of the theorem.\\

\end{proof}

\section{Zeros of  Polynomials.}

 Consider the polynomial of the type
$$P(z) = 1 + z^{h +a_1} + z^{2h + a_1 + a_2} +\cdots + z^{(m-1)h + a_1 + a_2 +\cdots + a_{m-1}}, \eqno (R)  $$
which appears in the generalized Riesz product connected with rank one measure preserving transformation.

It is easy to see that zeros of these polynomials cluster near the unit circle as $h$ tends to $\infty$. We prove a  quantitative result, namely, if $ w$ is a zero of this polynomial then

$$\Bigl(\frac{1}{2}\Bigr)^{\frac{1}{h}} \leq \mid w\mid \leq  (2)^{\frac{1}{h}} ~~~~~    \eqno (3)$$

To see this we write $\mid w \mid =a$. Assume first that $ a  \leq 1$. Then, since $w$ is a zero of $P$,
$$    a^h + a^{2h} + \cdots + a^{(m-1)h} \geq 1.$$
Equivalently,

$$a^h\frac{(1- a^{(m-1)h})}{1 - a^h} \geq 1. $$
$$a^h - a^{mh} \geq 1 - a^h$$
$$2a^h \geq 1 + a^{mh} \geq 1$$

which proves the result when $| w| \leq 1$. To prove the second half of the inequality we note that
if $\mid w \mid$ is greater than 1 then $\frac{1}{| w|} \leq 1$ and $\frac{1}{w}$ is a zero of $P(\frac{1}{z})$ so the second half follows from the first half. A slight improvement of the inequality is possible. If $m =2$ then all the zeros of $P$ lie on the unit circle. It is easy to show that if $m >2$ then the equation $x^m - 2x +1$ has a zero, say $b_m$, in the open interval $\frac{1}{2} < x < 1$. and one can show that

$$(b_m)^{\frac{1}{h}} \leq \mid w\mid  \leq \Big(\frac{1}{b_m}\Big)^{\frac{1}{h}}.$$
However, it is not a very big improvement since one can show that $b_m\rightarrow \frac{1}{2}$ as $m\rightarrow \infty$.

This simple result tells us that if each $P_k$ has less than $ch_{k-1}$ zeros bigger than 1 in absolute value
where $c$ is a positive constant less than one, then $\prod_{k=1}^\infty | \alpha_k| = 0$.\\

 We mention that M. Odlyzko and B. Poonen in \cite{Odlyzko} proved that the zeros of the polynomials with coefficients in $\{0,1\}$ are contained in the annulus $\frac1{\phi}<|z|<\phi$ where $\phi$ is the golden ratio.\\

Zeros of polynomials with restricted coefficients has deep and extensive literature. We mention only a result in a recent paper. (P. Brown, T. Erd\'elyi, F. Littmann \cite{BEL}). Let
$$K_n = \Biggl\{\sum_{k=0}^na_kz^k: \mid a_0\mid = \mid a_n\mid =1, \mid a_k\mid \leq 1\Biggr\},$$
and let $n$ be so large that $\delta_n = 33\pi\frac{log(n)}{\sqrt{n}} < 1,$
then any polynomial in $K_n$ admits at least $8\sqrt n \log n$ zeros in $\delta_n$ neighborhood of any point  of the unit the circle. Thus the derived set, i.e., the set of limit points of the zeros of the polynomials appearing in the generalized Riesz product (R) is the full unit circle.\\

 \paragraph{\textbf{Acknowledgements.}} The first author wishes to express his sincere
 thanks to the National Center for Mathematics and IIT, Mumbai, where a part of this paper was written, for an invitation and warm hospitality.\\


\begin{thebibliography}{99-}

\bibitem{elabdaletds}
 e. H. ~el Abdalaoui,{\em A new class of rank-one transformations with singular spectrum,} Ergodic Theory Dynam. Systems,
 \textbf{27} (2007), no. 5, 1541-1555.

\bibitem{elabdal-lem}
e. H. el Abdalaoui \& M. Leman\'czyk, {\em Approximately transitive dynamical systems and simple spectrum,} Arch. Math. (Basel),
 {\textbf{97}} (2011), no. 2, 187-197

\bibitem{elabdal}
e. H. el Abdalaoui, {\em On the singularity of the spectrum of rank one maps}, {preprint Feb. 2013.}

\bibitem{AH1}
C. Aistleitner \& M. Hofer, {\em On the maximal spectral type of a class of rank one
transformations,} Dyn. Syst., {\textbf{27 (4)}}, 2012, 515-523.

\bibitem{AH2}
C. Aistleitner, {\em On a problem of Bourgain concerning the $L^1$-norm of exponential sums,}
Math. Z., to appear.
 \bibitem{bor}
 P. Borwein, M. J. MNossinghoff, {\em Barker sequences and flat polynomials}, Number Theory and Polynomials, 71-88, Lond. Math. Soc. Lecture Notes Series, 352, Cambridge Univ Press, Cambridge, 2008.
 \bibitem{bor1}
 P. Borwein, M. J. MNossinghoff, {\em Wiefrich pairs and Barker sequences II}, Preprint, July 2013.
 \bibitem{Brown-Moran}
G.~Brown \& W.~ Moran, {\em On orthogonality of Riesz products,} Proc. Cambridge Philos. Soc., \textbf{76} (1974), 173-181.

\bibitem{BEL}
P. Brown, T. Erd\'elyi, F. Littmann,{\em Polynomials with coefficients from a finite
set}, to appear in Trans. Amer. Math. Soc.

\bibitem{bourgain}
J. Bourgain,\emph{ On the spectral type of Ornstein class one transformations}, Israel J.\ Math., {\textbf{ 84}} (1993), 53-63.

\bibitem{Choksi}
J. Choksi, M. G. Nadkarni {\em Maximal spectral type of a rank One transformation}, Canad. Math. Bull. ({\textbf{37 (1)}}) (1994), 29-36.

\bibitem{Chung}
K. L. Chung, {\em A course in probability theory,} Third edition. Academic Press, Inc., San Diego, CA, 2001.

\bibitem{Guenais}
M. ~Guenais, {\em Morse cocycles and simple Lebesgue spectrum } Ergodic Theory Dynam. Systems,
 \textbf{19} (1999), no. 2, 437-446.
\bibitem{down}
T. Downarowich, Y. Lacroix, {\em Merit Factors and Morse Sequences} Theoretical Computer Science, \textbf(209) (1998), 377-387.
\bibitem{Friedman}
N. Friedman, {\em Introduction to Ergodic Theory} van Nostrand-Reinhold, New York, 1970.

\bibitem{Hoffman}
 K.~Hoffman, {\em Banach spaces of analytic functions,} Reprint of the 1962 original. Dover Publications, Inc., New York, 1988.

\bibitem{Host-Mela-Parreau}
 B. Host, J.-F. M\'ela, F. Parreau, {\em Non-singular transformations and spectral analysis of measures}, Bull. Soc. math. France {\textbf{119}} (1991), 33-90.

\bibitem{Kahane}
J-P.~Kahane, \emph{Sur les polyn\^omes \`a coefficients unimodulaires,} (French) Bull. London Math. Soc., {\textbf{12}} (1980),
no. 5, 321-342.

\bibitem{KlemesR}
I. ~Klemes \& ~K. ~Reinhold, {\em  Rank one transformations with singular spectre type}, Isr. J. Math.,  \textbf{vol 98} (1997), 1-14.

\bibitem{Ledrappier}
F. Ledrappier,{\em Des produits de Riesz comme mesures spectrales}, Ann. Inst. H. Poincar\'e. Probab. Stat., {\textbf{4}} (1970), 335–344.

\bibitem{Nadkarni1}
M.~G.~Nadkarni, {\em Spectral Theory of Dynamical Systems },
Hindustan Book Agency, New Delhi, (1998); Birkh\"{a}user Advanced
Texts : Basler Lehrb{\AE}cher. [Birkh\"auser Advanced Texts: Basel
Textbooks] Birkh\"auser Verlag, Basel, 1998.

\bibitem{Nadkarni2}
M. G. Nadkarni {\em some remark on the spectrum of a rank one transformation} {preprint Feb. 2012}

\bibitem{Odlyzko}
A. M. Odlyzko and B. Poonen, {\em Zeros of polynomials with 0,1 coefficients,}  l'Enseign. Math., {\textbf{39}} (1993), pp. 317-348.

\bibitem{ornstein}
D. S. Ornstein, {\em On the root problem in ergodic  theory} Proc. Sixth Berkeley Symposium
on Mathematical Statistics and Probability, Berkeley and Los Angeles, Uiversity of California Press, {\textbf{vol 2}}{(1970)}, 348-356.
\bibitem{QS}
H.~Queffelec \& B. Saffari, {\em On Bernstein's inequality and Kahane's ultraflat polynomials,} J. Fourier Anal. Appl., {\textbf{2}} (1996), no. 6, 519-582.
\bibitem{Riesz}
F. Riesz, {\"{U}ber die Fourierkoeffizienten einer stetigen funktion von beschr\"{a}nkter Schwankung, M.Z.} {\textbf{2}}(1918), 312-315.
\bibitem{Zygmund} A. ~Zygmund, {Trigonometric series,} second ed.,
Cambridge Univ. Press, Cambridge, 1968.
\end{thebibliography}
\end{document}